\documentclass[msom,nonblindrev]{informs3}

\OneAndAHalfSpacedXI
\usepackage{subfigure}
\usepackage{hyperref}
\usepackage{fix-cm}
\usepackage{algorithm}
\usepackage{algorithmic}
\usepackage{booktabs}
\usepackage{tabularx}
\usepackage{array}
\usepackage{natbib}
 \bibpunct[, ]{(}{)}{,}{a}{}{,}%
 \def\bibfont{\small}%

\TheoremsNumberedThrough
\EquationsNumberedThrough

\makeatletter
\def\theARTICLETOP{}
\def\ps@arxivpreprint{%
  \def\@oddhead{}%
  \def\@evenhead{}%
  \def\@oddfoot{\hfil\thepage\hfil}%
  \def\@evenfoot{\hfil\thepage\hfil}%
  \def\sectionmark##1{}%
  \def\subsectionmark##1{}%
}
\makeatother
\begin{document}

\TITLE{Dynamic Cloud Service-Capacity Deployment with Costly Response Readiness}

\ARTICLEAUTHORS{
\AUTHOR{Ruihan Zhou}
\AFF{Guanghua School of Management, Peking University, Beijing 100871, China, \EMAIL{rhzhou@stu.pku.edu.cn}}

\AUTHOR{Xiaowei Zhang}
\AFF{Department of Industrial Engineering and Decision Analytics, The Hong Kong University of Science and Technology, Hong Kong SAR, China, \EMAIL{xiaoweiz@ust.hk}}

\AUTHOR{Yijie Peng}
\AFF{Department of AI for Management, School of Management and Engineering, School of Computer Science, Nanjing University, Nanjing 210008, China, \EMAIL{pengyijie@nju.edu.cn}}
} 

\ABSTRACT{ 
\textbf{Problem definition:}
AI product launches require service-capacity decisions before sufficient product-specific workload histories are available. Early service observations allow managers to update their beliefs about remaining workload. Unlike inventory, workload demand does not physically deplete service capacity, while preserving the ability to scale rapidly may require costly capacity-access and operational arrangements. We study how a firm should use evolving workload beliefs to decide how much capacity to deploy at the current review and whether to preserve future response readiness. \textbf{Methodology/results:}
We formulate a finite-horizon dynamic program with workload beliefs, serviceable capacity, and an absorbing response-channel state. A preservation--deployment decomposition evaluates the joint action by solving one deployment problem for each continuation side and comparing the optimized side values. The analysis identifies when the firm should preserve response readiness and how much capacity it should deploy conditional on that choice.  We translate this structure into the Cold-Start Belief Actionability Policy (CBAP), which estimates the side-specific values from workload scenarios drawn from available predictive beliefs. We provide bounds on the policy's decision and lifecycle performance losses. Across the tested instances, numerical experiments and a BurstGPT-based evaluation of the decision layer show that CBAP reduces
lifecycle cost relative to benchmark policies by avoiding premature capacity deployment and unnecessary readiness spending. \textbf{Managerial implications:}
The results separate forecast informativeness from operational
actionability and show that current deployment and future response readiness should be evaluated as distinct decision margins. Zero deployment can rationally represent Standby rather than Exit, while positive deployment need not imply preserving the response channel. Managers should therefore assess whether an updated workload belief supports current deployment, future response readiness, both, or neither.
}


\maketitle
    
\section{Introduction}
\label{sec:introduction}

\subsection{Background and Costly Response Readiness}
\label{subsec:intro_positioning}

The rapid commercialization of generative AI has made capacity planning a central issue in new product launches. For digital services, especially AI services, launch demand is converted almost immediately into inference workload, token usage, latency pressure, and cloud-capacity requirements. Firms have little product-specific workload history before launch and must rely on product attributes, lifecycles of related products, platform signals, beta-test traffic, and release schedules \citep{ZhouEtAl2026CDLF}. After launch, early service observations help the firm update its belief about remaining workload, but these observations depend on the capacity already assigned to the product. Capacity decisions therefore affect both current service quality and the information available for learning about future demand. Rapid response is important because workload may increase sharply after launch, while additional GPU capacity, cloud quotas, engineering support, and internal approval may not be immediately available. A firm that does not maintain these arrangements may be unable to respond quickly when stronger demand signals emerge.

We study product-level cloud service-capacity deployment for a cold-start AI product. The firm monitors workload and service metrics continuously but reviews capacity at discrete managerial epochs. At review epoch $r$, the state is $\mathcal S_r=(\mathbb P_r,I_r,o_r)$, where $\mathbb P_r$ is the workload belief, $I_r$ is the serviceable capacity assigned to the product, and $o_r$ indicates whether the rapid-response channel remains open. The firm chooses $(q_r,o_{r+1})$: $q_r$ is the additional service capacity deployed or activated at the current review, and $o_{r+1}$ determines whether rapid-response deployment remains feasible afterward. The firm updates its workload belief from observed service outcomes. When available capacity limits service, it observes $X_r(q_r)=\min\{D_r,I_r+q_r\}$ rather than latent workload demand $D_r$. The capacity state represents configured service capacity rather than consumable inventory. Workload affects service, utilization, shortfall, and belief updating, but does not physically deplete cloud servers or GPUs. Section~\ref{sec2} specifies the capacity transition and cost structure.

The central feature of this setting is that future responsiveness is costly. Rapid response during an AI product launch may require capacity reservations, priority-access arrangements, vendor standby, engineering monitoring, and pre-approved scaling procedures. These arrangements may need to be maintained even when the firm does not deploy additional capacity in the current review interval. Cloud capacity-assurance services provide an operational example. Amazon EC2 On-Demand Capacity Reservations and Capacity Blocks for ML allow customers to reserve compute or GPU capacity for current or future workloads, and charges may apply even when the reserved capacity is unused.\footnote{See Amazon Web Services, \emph{On-Demand Capacity Reservations}, \emph{Capacity Reservation pricing and billing}, and \emph{Capacity Blocks for ML}, \url{https://docs.aws.amazon.com/AWSEC2/latest/UserGuide/ec2-capacity-reservations.html}; \url{https://docs.aws.amazon.com/AWSEC2/latest/UserGuide/capacity-reservations-pricing-billing.html}; and \url{https://docs.aws.amazon.com/AWSEC2/latest/UserGuide/ec2-capacity-blocks.html}.}

We represent this distinction through the response-readiness cost $k_r$. The deployment cost $c_rq_r$ applies to current service-capacity deployment, the idle-capacity cost $h$ applies to configured capacity that remains unused, and the shortage or delay penalty $b$ applies to workload that available capacity cannot serve. In contrast, $k_r$ is the cost of future actionability: the firm incurs it when it preserves the response channel after the current review. The payment maintains future deployment feasibility but does not purchase current service capacity. If the firm does not pay $k_r$, the response channel enters an absorbing closed state, and rapid-response deployment is no longer feasible within the modeled launch-support window. Thus, $k_r$ differs from both a conventional idle-capacity cost and the fixed ordering cost $K\cdot\mathbf{1}\{q_r>0\}$ used in standard inventory models.

The response-channel state changes the decision itself. When the channel is open, zero deployment can mean either Exit, $(q_r=0,o_{r+1}=0)$, or Standby, $(q_r=0,o_{r+1}=1)$. A positive deployment can likewise be a Terminal Response, $(q_r>0,o_{r+1}=0)$, or a Sustained Response, $(q_r>0,o_{r+1}=1)$. Zero deployment may therefore incur a readiness cost, while positive deployment need not preserve future response. Once closed, the channel cannot be reopened within the modeled launch-support window, although existing capacity remains available. Figure~\ref{fig:motivating_inventory_comparison} contrasts this setting with a standard inventory model. Inventory models use demand information to choose replenishment for a stock that demand depletes. The AI cloud-capacity problem instead uses evolving workload beliefs to choose both current deployment and future response readiness for capacity that demand does not deplete.
\begin{figure}[htbp]
\centering
\includegraphics[width=0.9\textwidth]{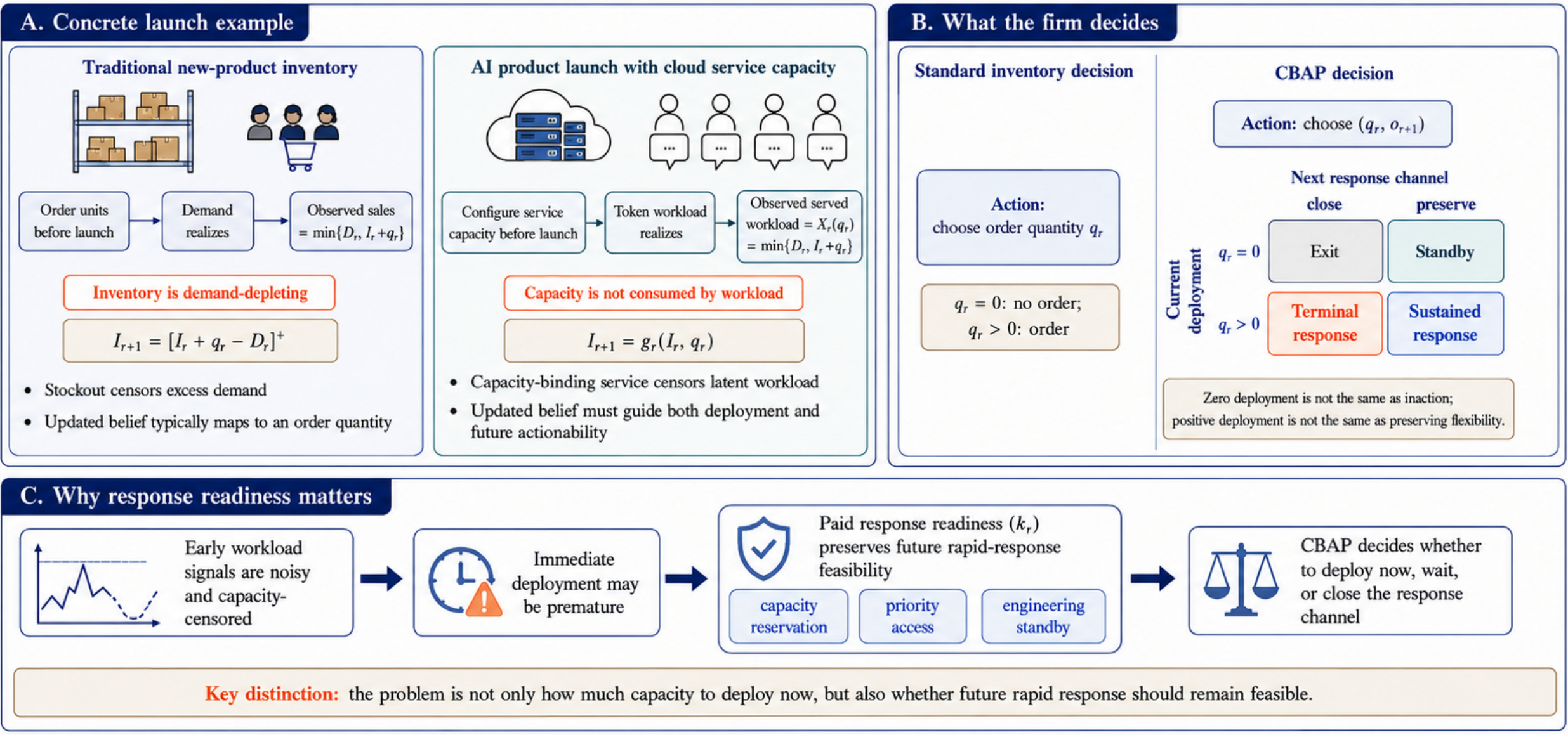}
\caption{Motivating contrast between inventory replenishment and response-readiness capacity deployment.}
\label{fig:motivating_inventory_comparison}
\end{figure}
\subsection{Related Literature and Research Gap}
\label{subsec:intro_literature}

\paragraph{Cloud service-capacity management under demand uncertainty.}

Cloud computing combines elastic resource provisioning with a layered value chain. Cloud providers plan and deploy infrastructure, while users choose among reserved, on-demand, and supplementary capacity options \citep{NISTMellGrance2011,ChenEtAl2023CloudValueChains}. At the user level, \citet{BulbulNoyanErol2021} formulate multistage provisioning decisions that balance reserved and on-demand resources. At the provider level, \citet{LiuEtAl2025CloudDeployment} study infrastructure preparation and server deployment under demand uncertainty and physical and temporal constraints. Related reservation models determine capacity commitments and supplementary contracts for intermittent demand surges \citep{ChenLeiMoinzadeh2024CloudReservation}. This literature establishes infrastructure planning, capacity sourcing, and reservation design as central cloud-operations problems.

AI inference introduces a distinct product-level problem. Workloads are heterogeneous and bursty, service performance depends on both workload composition and configured compute capacity, and rapid scaling may require assured capacity access rather than frictionless resource availability
\citep{ChenDeshpande2026AICloudValueChains}. Existing cloud-capacity studies have primarily examined infrastructure planning, capacity sourcing, and reservation decisions under demand uncertainty. They have paid less attention to product-level decisions that jointly determine current deployment and whether future rapid response remains feasible. We study these repeated decisions as workload beliefs evolve after launch. Workload does not physically deplete service capacity, and lower-level autoscaling, routing, and load balancing occur within the review interval and enter the model through service costs and capacity transitions.

\paragraph{Costly flexibility, reactive capacity, and response rights.}

Our response-readiness mechanism relates to the operations-management literature on costly flexibility and reactive capacity. Flexible resources and backup arrangements allow firms to defer some capacity commitments until uncertainty is resolved, but maintaining this ability generally requires an investment, flexibility premium, or reservation payment \citep{FineFreund1990,VanMieghem1998,JordanGraves1995}. Work on demand surges distinguishes safety stock from reactive capacity and shows the value of retaining the ability to respond after new information arrives \citep{HuangSongTong2016}. The supply-contract literature similarly separates payment for a future exercise right from the exercise decision itself, as in option contracts, backup agreements, and capacity-reservation arrangements \citep{EppenIyer1997,BarnesSchusterBassokAnupindi2002,SerelDadaMoskowitz2001,HazraMahadevan2009}. 

Our setting differs in three respects. Response readiness is neither a fixed flexible-capacity endowment nor a one-time option purchased at launch, and it does not commit the firm to current capacity. The firm reviews readiness repeatedly over the product lifecycle, may close the channel while retaining existing capacity, and chooses readiness separately from current deployment. It may therefore preserve future response without deploying immediately or deploy current capacity and then close the channel. Future actionability is consequently an endogenous operational state rather than an automatic consequence of holding a capacity option.

\paragraph{Inventory decisions with demand learning.}

Inventory models provide a useful comparison for decisions that use updated demand information. In lost-sales and newsvendor settings, firms choose replenishment for stock that realized demand depletes. When firms observe sales rather than latent demand, available inventory may censor the observation \citep{HuhRusmevichientong2009,HuhLeviRusmevichientongOrlin2011,BesbesMuharremoglu2013}. Data-driven inventory models also study how historical data, covariates, and demand observations can improve ordering decisions \citep{BanGallienMersereau2019}.

We use a related observation structure but study a different decision problem. Available service capacity may limit observed workload, so deployment can affect the information used to update the workload belief. We take this update as part of the model and do not propose a new method for censored learning. The main difference lies in the physical state and the decision dimensions. Inventory is depleted by demand, whereas service capacity evolves through provisioning, persistence, release, expiration, or scale-in. In addition to choosing current deployment, the firm decides whether future rapid deployment remains feasible. The response-channel decision therefore determines whether the firm can act on later workload information.

\paragraph{Research gap and positioning.}
Taken together, these research streams leave unresolved a product-level sequential decision: how should a firm determine both how much service capacity to deploy now and whether to continue paying to preserve rapid deployment later? Existing studies examine capacity deployment, costly flexibility, and demand learning largely as separate decisions. To our knowledge, they have not jointly modeled repeated deployment and endogenous response readiness as workload beliefs evolve after launch. This separation is consequential because preserving the ability to act later is distinct from deploying capacity now. It creates four operating modes and makes future response readiness an endogenous operational state.

\subsection{Methodology and Contributions}
\label{subsec:intro_method_evidence}
We formulate a finite-horizon dynamic program with workload beliefs, serviceable capacity, and an absorbing response-channel state. Our contributions are threefold.

First, we introduce costly response readiness as an endogenous and repeatedly renewed operational decision. Unlike formulations in which flexibility is fixed in advance or follows automatically from a capacity commitment, our model separates the payment that preserves future deployment feasibility from the decision that deploys current service capacity. The firm may therefore preserve the response channel without deploying capacity, or deploy capacity and then close the channel. This separation identifies response readiness as a distinct decision margin and generates the four operating modes described above.

Second, we establish the structural form of the joint readiness--deployment decision. The preservation--deployment
decomposition solves one deployment problem for each  continuation side and then compares the optimized side values. Under scalar-belief regularity conditions, this decomposition yields one preservation boundary, two side-specific deployment boundaries, and a conditional base-capacity rule. These results distinguish the decision of whether future action should remain feasible from the decision of how much capacity should be deployed conditional on that choice.

Third, we translate this structure into the Cold-Start Belief
Actionability Policy (CBAP). CBAP estimates the two side values from workload scenarios and applies the same preservation--deployment comparison online. It is agnostic to the source of the predictive belief and can use posterior distributions, empirical scenarios, or generative lifecycle distributions. We provide a local decision-error bound and an expected lifecycle error bound relative to the grid-restricted optimum, separating errors due to belief approximation, scenario sampling, continuation-value evaluation, and numerical implementation. Numerical experiments and a BurstGPT-based application evaluate this decision layer. Across the tested instances, CBAP reduces lifecycle cost relative to the benchmark policies by avoiding premature deployment and
unnecessary readiness spending, and all four operating modes arise on BurstGPT-derived workload paths.

The remainder of the paper is organized as follows. Section~\ref{sec2} presents the model, Section~\ref{sec3} characterizes the response-readiness mechanism, and Section~\ref{sec4} develops CBAP and its decision-error bounds. Section~\ref{sec5} reports the numerical experiment and BurstGPT application. Section~\ref{sec6} concludes.
\section{Problem Formulation}
\label{sec2}
We consider service-capacity deployment over a finite launch-support window for a cold-start AI product. The firm reviews capacity at epochs $r=1,\ldots,R$, and $D_r$ denotes workload aggregated over the interval between reviews. The finite horizon represents the review and evaluation window used for the launch analysis rather than the permanent economic termination of the product or its cloud-service arrangements. Review $R$ is therefore the final review included in the reported lifecycle accounting. At each post-launch review, the firm observes its workload belief, serviceable capacity, and response-channel state before choosing current deployment and future response feasibility. Before launch, it chooses initial deployment $q_0\ge0$ at unit cost $c_0$. If the response channel is open at review $r$, the firm may deploy additional capacity $q_r\in\mathcal Q_r=[0,\bar q_r]$. Capacity is measured in workload-equivalent units, such as token-serving capacity, GPU-hour capacity, or another normalized compute measure. The objective includes initial and post-launch deployment, response readiness, idle capacity, and shortage or latency costs.

The baseline model assumes that capacity deployed at a review becomes serviceable within the same review interval. A deterministic lead time can be incorporated by adding pipeline capacity to the state. This extension would shift the policy boundaries but would not change the distinction between current deployment and future response feasibility. Review-dependent unit costs $c_r$ capture the trade-off between the information gained by waiting and the potentially higher cost of later deployment, as in multiordering and demand-update sourcing models \citep{WangAtasuKurtulus2012,FedergruenLiuLu2026}.

\subsection{Demand, Belief, and System State}

Let $C$ denote the cold-start information available before launch. Let $D_{1:R}=(D_1,\ldots,D_R)$ be the latent potential demand path, where $D_r\ge0$ is potential demand during the interval after review $r$. It may represent inference requests, token usage, GPU-hour demand, or another measure of compute workload aggregated over the review interval.

Let $I_r$ denote serviceable capacity at the beginning of review epoch $r$, before the deployment decision. If the firm deploys $q_r$, the capacity available during the review interval is $I_r+q_r$. Served demand is
\begin{equation}
\label{eq:served_demand}
X_r(q_r)=\min\{D_r,I_r+q_r\}.
\end{equation}
When $X_r(q_r)<I_r+q_r$, served demand equals potential demand. When $X_r(q_r)=I_r+q_r$, the firm learns only that potential demand is at least the available capacity. Belief updating therefore uses served demand together with the known capacity limit. Unused service capacity and unmet demand are
\begin{equation}
\label{eq:service_outcomes}
U_r(q_r)=[I_r+q_r-D_r]^+,
\qquad
B_r(q_r)=[D_r-I_r-q_r]^+,
\end{equation}
where $[x]^+=\max\{x,0\}$. Here, $U_r(q_r)$ is deployed but unused capacity, and $B_r(q_r)$ is unmet demand, delayed requests, lost usage, or another shortage measure.

The function $g_r(\cdot,\cdot)$ captures persistence, expiration, release, automatic scale-in, or internal resource reallocation after the review interval. A parsimonious specification is
\begin{equation}
\label{eq:baseline_capacity_transition}
g_r(I_r,q_r)=(1-\delta_r)(I_r+q_r),
\qquad
\delta_r\in[0,1],
\end{equation}
where $\delta_r$ is an exogenous capacity-release parameter. The transition $g_r$ is an operational capacity-management rule, not a demand-depletion equation. The state $I_r$ is a product-level service-capacity budget measured in workload-equivalent units rather than a physical GPU count. It aggregates heterogeneous instances, regional pools, and lower-level autoscaling decisions. Non-depleting capacity need not be permanent: it may decay, expire, be released, or be reassigned through $g_r$. The essential distinction from inventory is that realized workload affects service outcomes and belief updating but is not mechanically subtracted from the next capacity state.

At the beginning of review epoch $r$, after observing past deployment decisions, served demand, and capacity states, let $\mathbb P_r$ denote the belief state over the remaining demand path:
\begin{equation}
\label{eq:demand_belief}
\mathbb P_r
=
\mathbb P
\left(
D_{r:R}
\mid
C,
q_{0:r-1},
X_{1:r-1},
I_{1:r}
\right).
\end{equation}
Sections~\ref{sec2} and \ref{sec3} use $\mathbb P_r$ as the belief state in the theoretical dynamic program. Section~\ref{sec4} introduces the learned or estimated belief used by CBAP. The belief may be represented by a posterior distribution, scenario tree, generative lifecycle model, or another predictive distribution over $D_{r:R}$. For example, conditional diffusion models can construct
cold-start lifecycle distributions by combining static product
descriptors, reference trajectories from related products, and newly
arriving observations \citep{ZhouEtAl2026CDLF}. The present
formulation requires only an available conditional predictive
distribution and is otherwise agnostic to the belief-generation
method.

Let $o_r\in\{0,1\}$ denote the response-channel state at the beginning of review epoch $r$. The response channel is available when $o_r=1$ and has been closed when $o_r=0$. Closure is irreversible. The full system state is
\begin{equation}
\label{eq:full_state}
\mathcal S_r
=
\left(
\mathbb P_r,
I_r,
o_r
\right).
\end{equation}
Unless stated otherwise, the response channel is initially available, so $o_1=1$. At review $r$, the firm first observes $\mathcal S_r$. If the response channel is open, it then chooses deployment $q_r$ and the next channel state $o_{r+1}$. If the channel is closed, no additional deployment is feasible and the channel remains closed. Demand $D_r$ is then realized, service outcomes are determined, and capacity evolves according to \eqref{eq:baseline_capacity_transition}. The firm uses the resulting observation to update $\mathbb P_{r+1}$ and enter state $\mathcal S_{r+1}$.

\subsection{Action Set and Cost Structure}

The admissible action depends on the current response-channel state. If $o_r=1$, the firm chooses $q_r\in\mathcal Q_r$ and $o_{r+1}\in\{0,1\}$. If $o_r=0$, the channel is closed and absorbing within the current launch lifecycle, so rapid-response deployment is no longer feasible under the pre-arranged review cadence. Formally,
\begin{equation}
\label{eq:action_set}
\mathcal A_r(o_r)
=
\begin{cases}
\{(q_r,o_{r+1}):q_r\in\mathcal Q_r,\ o_{r+1}\in\{0,1\}\}, & o_r=1,\\[3pt]
\{(0,0)\}, & o_r=0.
\end{cases}
\end{equation}
When $o_r=1$, the pair $(q_r,o_{r+1})$ induces four action modes:
\begin{itemize}
\item \textbf{\emph{Exit:}} $(q_r=0,o_{r+1}=0)$. The firm deploys no additional capacity, closes the response channel, and relies on the existing capacity trajectory.

\item \textbf{\emph{Standby:}} $(q_r=0,o_{r+1}=1)$. The firm deploys no additional capacity but maintains priority access, engineering monitoring, or budget authorization to preserve future response feasibility.

\item \textbf{\emph{Terminal Response:}} $(q_r>0,o_{r+1}=0)$. The firm deploys additional capacity in the current interval and then closes the channel, making this deployment the final rapid response within the modeled window.

\item \textbf{\emph{Sustained Response:}} $(q_r>0,o_{r+1}=1)$. The firm deploys additional capacity in the current interval and preserves the channel so that it can respond again after later workload observations.
\end{itemize}

The absorbing-channel assumption is a launch-window governance abstraction; it does not mean that the firm can never procure cloud resources again. After releasing launch-specific priority access, engineering standby, budget authorization, or pre-approved scaling procedures, the firm cannot restore rapid deployment within the same review cadence without a new procurement or escalation cycle. Because such a restart lies outside the modeled window, the model represents closure as absorbing. Closure leaves existing serviceable capacity unchanged but removes future rapid-response deployment within that window. Figure~\ref{fig:two_layer_response_flow} summarizes the action set and lifecycle flow.
\begin{figure}[htbp]
\centering
\includegraphics[width=0.9\textwidth]{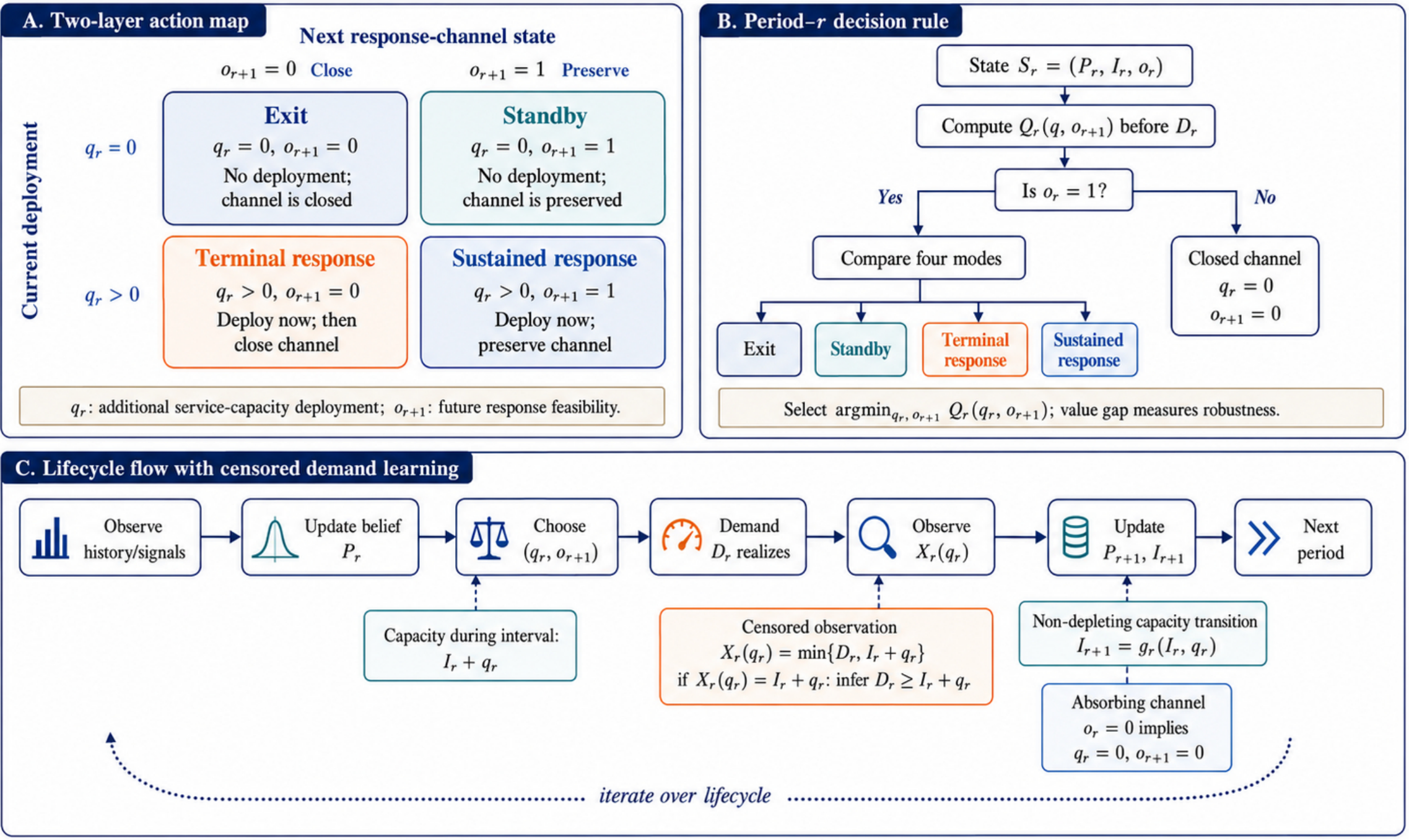}
\caption{Response action set and lifecycle decision flow.}
\label{fig:two_layer_response_flow}
\end{figure}

The firm incurs response-readiness cost $k_r$ when it preserves the channel, with $o_{r+1}=1$ as the accounting indicator. This rule applies at every scheduled review, including review $R$. The cost is tied to response feasibility rather than the current deployment amount: Exit and Terminal Response do not incur $k_r$, whereas Standby and Sustained Response do. At review $R$, the payment represents the cost of maintaining readiness through the final evaluation interval; it does not introduce a separately modeled decision at review $R+1$.

For a current deployment quantity $q_r$, define the operating cost
\begin{equation}
\label{eq:stage_operating_cost}
\ell_r(I_r,q_r,D_r)
=
h[I_r+q_r-D_r]^+
+
b[D_r-I_r-q_r]^+.
\end{equation}
Here, $h\ge0$ is the unit cost of deployed but unused capacity, and $b\ge0$ is the unit shortage, delay, or lost-demand penalty. Deployment and response-readiness costs are $c_rq_r$ and $k_ro_{r+1}$, respectively.

For notational simplicity, $I_1$ includes initial deployment $q_0$, while its acquisition cost remains explicit. A feasible policy $\pi$ maps each state $\mathcal S_r$ to an admissible pair $(q_r,o_{r+1})$. Its realized lifecycle cost is
\begin{equation}
\label{eq:objective_cost}
\mathrm{Cost}^{\pi}
=
c_0q_0^\pi
+
\sum_{r=1}^{R}
\left[
c_rq_r^\pi
+
k_ro_{r+1}^\pi
+
\ell_r
\left(
I_r^\pi,q_r^\pi,D_r
\right)
\right].
\end{equation}
The objective separates current deployment, readiness preservation, and realized service outcomes:
\begin{equation}
\label{eq:policy_objective}
\min_{\pi\in\Pi}
\mathbb E
\left[
\mathrm{Cost}^{\pi}
\mid C
\right],
\end{equation}
where $\Pi$ is the set of admissible policies and the expectation is taken over the latent demand path and the observations generated by the service process.

\subsection{Dynamic Program}

Let $V_r^1(\mathbb P_r,I_r)$ and $V_r^0(\mathbb P_r,I_r)$ denote the minimum expected costs from review $r$ onward when the response channel is open and closed, respectively. They form the two channel-state layers of the augmented-state dynamic program. For deployment quantity $q$, define $I_{r+1}(q)=g_r(I_r,q)$, and let $\mathbb P_{r+1}(q)$ be the updated belief after observing the service outcome generated during interval $r$. This notation makes explicit that deployment affects both the next capacity state and the observation used for belief updating.

We set $V_{R+1}^0=V_{R+1}^1=0$ and impose standard regularity conditions: the deployment sets $\mathcal Q_r$ are compact; stage costs are measurable and integrable; the capacity and belief transitions are measurable kernels; and the side-specific objectives are lower semicontinuous in the deployment quantity. The zero boundary values truncate costs and decisions outside the evaluation window; they do not assert that the product or its operational arrangements have zero economic value after review $R$. These conditions ensure finite Bellman expectations and measurable minimizers. The pre-launch deployment set is also a compact operational interval, $\mathcal Q_0=[0,\bar q_0]$.

When the response channel is closed, no additional deployment is feasible. The closed-channel value satisfies
\begin{equation}
\label{eq:closed_bellman}
V_r^0(\mathbb P_r,I_r)
=
\mathbb E_{\mathbb P_r}
\left[
\ell_r(I_r,0,D_r)
+
V_{r+1}^0
\left(
\mathbb P_{r+1}(0),
I_{r+1}(0)
\right)
\mid
\mathbb P_r,I_r
\right].
\end{equation}
After closure, the belief may continue to update and existing capacity may evolve through $g_r(I_r,0)$, but the absorbing channel state prevents further deployment.

When the channel is open, the firm may close or preserve it after the current interval. For a given $q\in\mathcal Q_r$, define the close-after-current quantity value
\begin{equation}
\label{eq:closed_after_current_value}
O_r^0(q;\mathbb P_r,I_r)
=
c_rq
+
\mathbb E_{\mathbb P_r}
\left[
\ell_r(I_r,q,D_r)
+
V_{r+1}^0
\left(
\mathbb P_{r+1}(q),
I_{r+1}(q)
\right)
\mid
\mathbb P_r,I_r
\right],
\end{equation}
and the preserve-after-current quantity value
\begin{equation}
\label{eq:preserved_after_current_value}
O_r^1(q;\mathbb P_r,I_r)
=
c_rq
+
\mathbb E_{\mathbb P_r}
\left[
\ell_r(I_r,q,D_r)
+
V_{r+1}^1
\left(
\mathbb P_{r+1}(q),
I_{r+1}(q)
\right)
\mid
\mathbb P_r,I_r
\right].
\end{equation}
In $O_r^a$, $a\in\{0,1\}$ denotes the next channel state $o_{r+1}=a$, not the current state.

The open-channel value satisfies
\begin{equation}
\label{eq:active_bellman}
V_r^1(\mathbb P_r,I_r)
=
\min
\left\{
\min_{q\in\mathcal Q_r}
O_r^0(q;\mathbb P_r,I_r),
\;
k_r+
\min_{q\in\mathcal Q_r}
O_r^1(q;\mathbb P_r,I_r)
\right\}.
\end{equation}
The first term in \eqref{eq:active_bellman} gives the best current deployment when the firm closes the channel after the current interval. The second gives the best deployment when the firm pays $k_r$ to preserve rapid-response feasibility. The next channel state is therefore a separate decision rather than a passive consequence of deployment.

The pre-launch deployment can be optimized jointly with the post-launch policy by setting
$I_1=q_0$ and 
$\mathbb P_1=\mathbb P(D_{1:R}\mid C)$, then solving
\begin{equation}
\label{eq:initial_deployment}
q_0^\star
\in
\arg\min_{q_0\in\mathcal Q_0}
\left\{
c_0q_0
+
\mathbb E
\left[
V_1^1(\mathbb P_1,I_1=q_0)
\mid C
\right]
\right\}.
\end{equation}
The recursion in \eqref{eq:active_bellman} separates current deployment from future response feasibility, which is the basis for the policy structure developed next.
 
\section{Response Readiness and Deployment Structure}
\label{sec3}

This section shows how readiness cost separates future response feasibility from current capacity deployment. Section~\ref{sec:preservation_deployment_comparison} decomposes the joint action, Section~\ref{sec:threshold_structure} derives scalar-belief boundaries, and Section~\ref{sec:base_capacity_representation} gives a conditional base-capacity representation. Appendix~\ref{app:sec3_proofs} provides the proofs, and Appendix~\ref{app:always_open_readiness_loss} bounds the cost of a full-window commitment policy.

\subsection{Preservation--Deployment Decomposition}
\label{sec:preservation_deployment_comparison}

We rewrite the open-channel Bellman equation using side-specific values. This decomposition is an evaluation device; it does not imply that preservation and deployment occur in a separate physical sequence. At an open-channel state, the firm solves one conditional deployment problem for each continuation side and compares the optimized values. For $q\in\mathcal Q_r$, let $O_r^0(q;\mathbb P_r,I_r)$ denote expected cost, excluding the current readiness payment, when the channel is closed after the current interval. Define $O_r^1(q;\mathbb P_r,I_r)$ analogously for preservation.

Define
\begin{equation}
\label{eq:A0_A1_definitions}
A_r^a(\mathbb P_r,I_r)
=
\min_{q\in\mathcal Q_r}
O_r^a(q;\mathbb P_r,I_r),
\qquad a\in\{0,1\},
\end{equation}
and let
\begin{equation}
\label{eq:q0_q1_definitions}
q_r^{a*}(\mathbb P_r,I_r)
\in
\arg\min_{q\in\mathcal Q_r}
O_r^a(q;\mathbb P_r,I_r),
\qquad a\in\{0,1\},
\end{equation}
where ties are resolved by selecting the smallest minimizer. The gross preservation value is
\begin{equation}
\label{eq:gross_option_value}
G_r(\mathbb P_r,I_r)
=
A_r^0(\mathbb P_r,I_r)
-
A_r^1(\mathbb P_r,I_r).
\end{equation}
The quantity $G_r(\mathbb P_r,I_r)$ is the gross cost reduction from preserving the response channel before paying $k_r$.

Using \eqref{eq:gross_option_value}, the open-channel Bellman equation in \eqref{eq:active_bellman} can be written as
\begin{equation}
\label{eq:compact_bellman}
V_r^1(\mathbb P_r,I_r)
=
\min
\left\{
A_r^0(\mathbb P_r,I_r),
\;
k_r+
A_r^1(\mathbb P_r,I_r)
\right\}.
\end{equation}
The comparison between $A_r^0(\mathbb P_r,I_r)$ and $k_r+A_r^1(\mathbb P_r,I_r)$ determines whether the firm closes or preserves the channel. With ties resolved in favor of closing, closure occurs when $G_r(\mathbb P_r,I_r)\le k_r$, and preservation occurs when $G_r(\mathbb P_r,I_r)>k_r$. The firm then implements the deployment quantity for the selected side:
\begin{equation}
\label{eq:preservation_deployment_rule}
\begin{cases}
o_{r+1}^*=0,\quad q_r^*=q_r^{0*}(\mathbb P_r,I_r), & \text{if } G_r(\mathbb P_r,I_r)\le k_r,\\[3pt]
o_{r+1}^*=1,\quad q_r^*=q_r^{1*}(\mathbb P_r,I_r), & \text{if } G_r(\mathbb P_r,I_r)> k_r.
\end{cases}
\end{equation}
This representation separates readiness from deployment. The cost $k_r$ does not directly penalize current deployment; it enters only the comparison between the two optimized continuation sides. Combining the selected side with whether its conditional deployment quantity is zero yields the four modes defined in Section~\ref{sec2}: Exit, Terminal Response, Standby, and Sustained Response.

The comparison between Exit and Standby further separates readiness from current operating cost. Both choices set $q_r=0$, so they have the same current deployment cost, operating term, and next capacity state. Their difference is
\begin{align}
&
\left[
k_r+O_r^1(0;\mathbb P_r,I_r)
\right]
-
O_r^0(0;\mathbb P_r,I_r)
\nonumber\\
&=
k_r
-
\mathbb E_{\mathbb P_r}
\left[
V_{r+1}^0
\left(
\mathbb P_{r+1}(0),
I_{r+1}(0)
\right)
-
V_{r+1}^1
\left(
\mathbb P_{r+1}(0),
I_{r+1}(0)
\right)
\mid
\mathbb P_r,I_r
\right].
\label{eq:standby_exit_cancellation}
\end{align}
Because $\ell_r(I_r,0,D_r)$ cancels, Standby is preferred to Exit only when the expected future value of feasible response actions exceeds the readiness cost. The next belief in \eqref{eq:standby_exit_cancellation} is generated by the served-demand observation under $q_r=0$ and the known capacity limit $I_r$.

\subsection{Scalar-Belief Boundaries for Readiness and Deployment}
\label{sec:threshold_structure}

Because the full belief $\mathbb P_r$ may be high dimensional, we use a scalar-belief benchmark to obtain explicit decision boundaries. Let $m_r\in\mathcal M_r$, where $\mathcal M_r\subseteq\mathbb R$ is an interval, summarize the belief about remaining workload. A larger $m_r$ represents a predictive distribution with stronger current and remaining workload. The index may be a posterior mean, burst probability, latent demand-scale parameter, or another statistic that orders predictive workload distributions.

For $a\in\{0,1\}$, define the scalar quantity value
\begin{equation}
\label{eq:scalar_quantity_value_by_channel}
O_r^a(q;m,I)
=
c_rq
+
\mathbb E
\left[
\ell_r(I,q,D_r)
+
V_{r+1}^a
\left(
m_{r+1}(q),
I_{r+1}(q)
\right)
\mid m
\right],
\end{equation}
where $a=0$ denotes closure after the current interval, $a=1$ denotes preservation, $I_{r+1}(q)=g_r(I,q)$ is the non-depleting capacity transition, and $m_{r+1}(q)$ is the next scalar belief generated by the served-demand observation and the known capacity limit $I+q$. The optimized side values and conditional deployment quantities are
\begin{equation}
\label{eq:scalar_q_definitions}
A_r^a(m,I)
=
\min_{q\in\mathcal Q_r}
O_r^a(q;m,I),
\qquad
q_r^{a*}(m,I)
\in
\arg\min_{q\in\mathcal Q_r}
O_r^a(q;m,I),
\qquad a\in\{0,1\}.
\end{equation}
When the minimizer is not unique, the smallest minimizer is selected. The scalar gross preservation value is
\[
G_r(m,I)=A_r^0(m,I)-A_r^1(m,I).
\]
This value measures the cost reduction from preservation before payment of the current readiness cost $k_r$.

The scalar-boundary result is not needed for either the dynamic program or the scenario-based implementation in Section~\ref{sec4}. It shows that, when a single index orders the belief, one preservation boundary and two conditional deployment boundaries represent the four modes. The result requires the convexity and single-crossing conditions stated below and does not assert that they hold for every belief-updating model.

\begin{assumption}[Quantity-value regularity]
\label{ass:quantity_value_regular}
For each $a\in\{0,1\}$, the feasible set $\mathcal Q_r=[0,\bar q_r]$ is compact with $\bar q_r>0$, and $O_r^a(q;m,I)$ is continuous in $(q,m)$, convex in $q$, and has decreasing differences in $(q,m)$. The right marginal value of deployment at zero,
\begin{equation}
\label{eq:psi_deployment_margin}
\psi_r^a(m;I)
=
\left.
\frac{\partial O_r^a(q;m,I)}
{\partial q}
\right|_{q=0^+},
\qquad a\in\{0,1\},
\end{equation}
exists, is continuous in $m$, and is strictly decreasing in $m$.
\end{assumption}
Under this assumption, conditional deployment has a threshold structure: on side $a$, positive deployment becomes optimal when the zero-deployment margin $\psi_r^a(m;I)$ becomes negative.

\begin{assumption}[Preservation-value single crossing]
\label{ass:preservation_single_crossing}
For fixed $(r,I)$ with $r<R$, the scalar preservation value $G_r(m,I)$ is continuous and nondecreasing in $m$.
\end{assumption}
This assumption gives channel preservation a threshold structure because stronger workload beliefs weakly increase the gross value of future response feasibility. Under Assumption~\ref{ass:quantity_value_regular}, the smallest conditional optimizer is nondecreasing in $m$ on each side. Stronger workload beliefs therefore make deployment weakly more attractive, although the implemented quantity still depends on the selected continuation side.

\begin{assumption}[Interior crossings]
\label{ass:interior_crossings}
For fixed $(r,I)$ with $r<R$, there exist $\underline m_r^P<\overline m_r^P$ in $\mathcal M_r$ such that $G_r(\underline m_r^P,I)<k_r<G_r(\overline m_r^P,I)$, and $G_r(\cdot,I)$ is strictly increasing on $[\underline m_r^P,\overline m_r^P]$. For each $a\in\{0,1\}$, there exist $\underline m_r^a<\overline m_r^a$ in $\mathcal M_r$ such that $\psi_r^a(\underline m_r^a;I)>0>\psi_r^a(\overline m_r^a;I)$.
\end{assumption}
This assumption restricts the analysis to the nondegenerate case in which the preservation and deployment thresholds lie inside the scalar belief space.

\begin{proposition}[Scalar-belief boundaries]
\label{prop:scalar_boundaries}
For $r<R$, under Assumptions~\ref{ass:quantity_value_regular}--\ref{ass:interior_crossings}, the smallest conditional optimizer $q_r^{a*}(m,I)$ is nondecreasing in $m$ for each $a\in\{0,1\}$. There also exist unique thresholds $m_r^P(k_r;I)$, $m_r^{D,0}(I)$, and $m_r^{D,1}(I)$ satisfying 
$G_r(m_r^P(k_r;I),I)=k_r$, $\psi_r^0(m_r^{D,0}(I);I)=0$, and 
$\psi_r^1(m_r^{D,1}(I);I)=0$. With ties resolved in favor of closing at the preservation boundary and zero deployment at the deployment boundaries, the optimal open-channel action is
\begin{equation}
\label{eq:four_region_policy}
\begin{cases}
\text{Exit}, & m_r \le m_r^P(k_r;I) \text{ and } m_r \le m_r^{D,0}(I),\\[3pt]
\text{Terminal Response}, & m_r \le m_r^P(k_r;I) \text{ and } m_r > m_r^{D,0}(I),\\[3pt]
\text{Standby}, & m_r > m_r^P(k_r;I) \text{ and } m_r \le m_r^{D,1}(I),\\[3pt]
\text{Sustained Response}, & m_r > m_r^P(k_r;I) \text{ and } m_r > m_r^{D,1}(I).
\end{cases}
\end{equation}
Holding $k_{r+1},\ldots,k_R$ and all other model primitives fixed, the current readiness cost $k_r$ shifts only the preservation boundary. In particular, if $k_r'>k_r$ and both corresponding preservation thresholds exist, then
\begin{equation}
\label{eq:preservation_boundary_monotonicity}
m_r^P(k_r';I)>m_r^P(k_r;I),
\end{equation}
whereas $m_r^{D,0}(I)$ and $m_r^{D,1}(I)$ are unaffected by current $k_r$.
\end{proposition}

\noindent
The proof is provided in Appendix~\ref{app:proof_scalar_boundaries}. Figure~\ref{fig:scalar_belief_action_partition} illustrates the four regions in Proposition~\ref{prop:scalar_boundaries}. A stronger workload belief makes deployment more attractive on either continuation side but does not alone determine whether the firm preserves future response. Preservation depends on the comparison between gross value $G_r(m,I)$ and readiness cost $k_r$. Zero deployment may therefore represent either Exit or paid Standby, while positive deployment may represent either Terminal Response or Sustained Response. The firm evaluates the same workload belief on both continuation sides before selecting a joint action.

\begin{figure}[htp]
\centering
\includegraphics[width=0.7\textwidth]{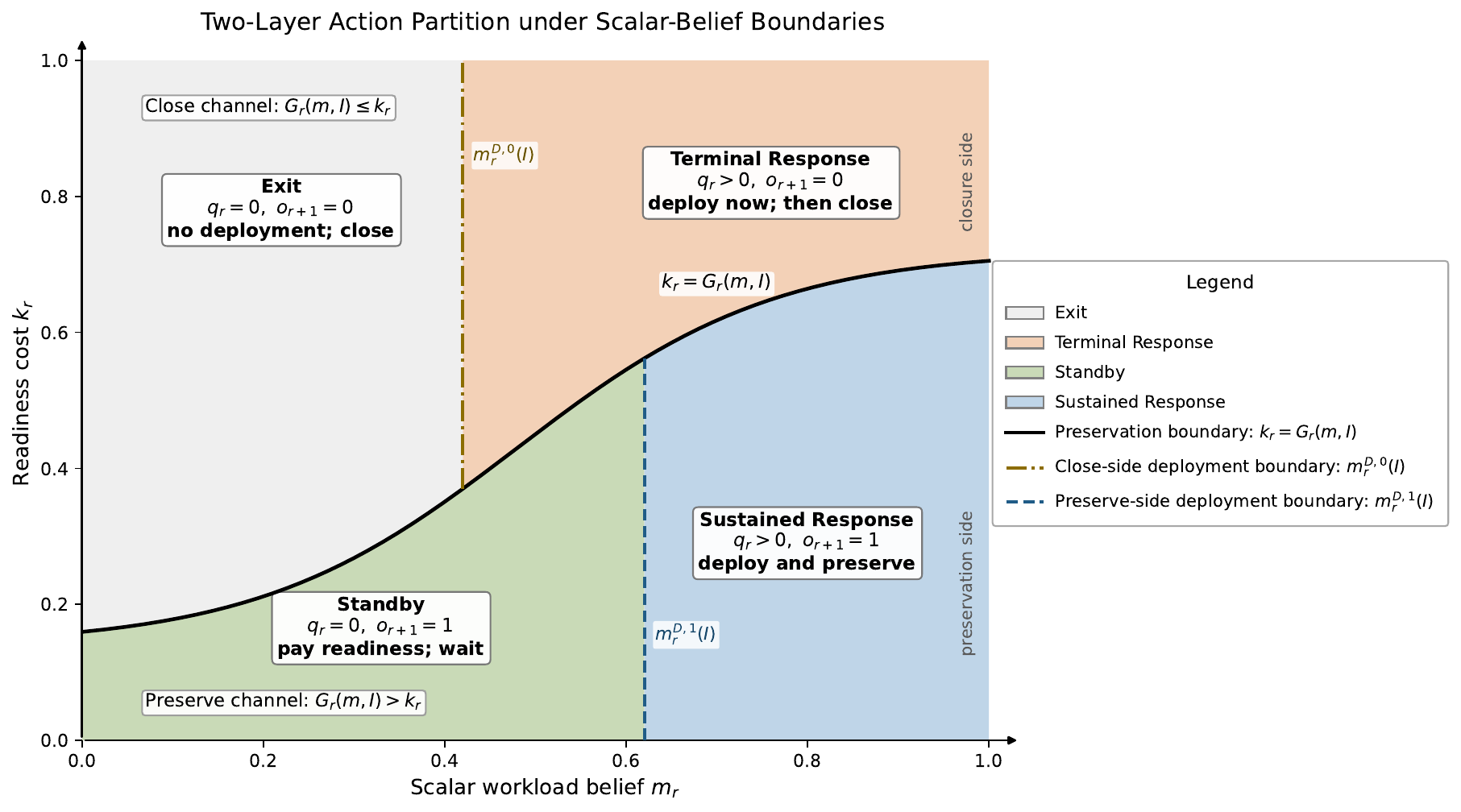}
\caption{Actionability regions under scalar beliefs.}
\label{fig:scalar_belief_action_partition}
\end{figure}

\paragraph{Conditional Base-Capacity Representation}
\label{sec:base_capacity_representation}

We next characterize the deployment quantity after the response-readiness comparison has selected a continuation side. This result connects the model to classical capacity and inventory policies. Let $y=I+q$ denote the service-capacity budget available during interval $r$ after deployment. Suppose the post-interval capacity state depends on deployment through this post-deployment capacity level, so that
\[
g_r(I,q)=\widetilde g_r(I+q)=\widetilde g_r(y).
\]
Let $X_r(y)=\min\{D_r,y\}$ be served workload, and let
\[
m_{r+1}(y)=\mathcal B_r(m,X_r(y),y)
\]
be the next scalar belief generated by served workload and the known capacity limit. For continuation side $a\in\{0,1\}$, define
\begin{equation}
\label{eq:Phi_base_capacity}
\Phi_r^a(y;m)
=
c_ry
+
\mathbb E_m
\left[
h(y-D_r)^+
+
b(D_r-y)^+
+
V_{r+1}^a
\left(
m_{r+1}(y),
\widetilde g_r(y)
\right)
\right].
\end{equation}
For $y=I+q$, this gives
\[
O_r^a(q;m,I)
=
-c_rI+\Phi_r^a(I+q;m).
\]

\begin{assumption}[Conditional base-capacity regularity]
\label{ass:base_capacity}
Retain the scalar-belief benchmark in Section~\ref{sec:threshold_structure}. For each review epoch $r$, scalar belief $m$, and continuation side $a\in\{0,1\}$, $\Phi_r^a(\cdot;m)$ is proper, lower semicontinuous, and convex on $[0,\infty)$, and it attains a minimum on this domain. The predictive distribution of $D_r$ under $m$ has a continuous distribution function $F_r(\cdot\mid m)$. The continuation term
\begin{equation}
\label{eq:Xi_base_capacity}
\Xi_r^a(y;m)
=
\mathbb E_m
\left[
V_{r+1}^a
\left(
m_{r+1}(y),
\widetilde g_r(y)
\right)
\right]
\end{equation}
is convex on $[0,\infty)$ and differentiable at the interior capacity levels used below. At kink points, its convex subdifferential is nonempty.
\end{assumption}

\begin{corollary}[Conditional base-capacity rule]
\label{cor:base_capacity}
Under Assumption~\ref{ass:base_capacity}, define the side-specific base-capacity target by
\begin{equation}
\label{eq:base_capacity_target}
\bar y_r^a(m)
\in
\arg\min_{y\ge 0}
\Phi_r^a(y;m),
\qquad a\in\{0,1\},
\end{equation}
with the smallest minimizer selected when the minimizer is not unique. Conditional on continuation side $a$, the optimal deployment quantity is
\begin{equation}
\label{eq:conditional_base_capacity_rule}
q_r^{a*}(m,I)
=
\min
\left\{
\bar q_r,
\left[\bar y_r^a(m)-I\right]^+
\right\},
\qquad a\in\{0,1\}.
\end{equation}
If there is no upper deployment bound, then
\[
q_r^{a*}(m,I)
=
\left[\bar y_r^a(m)-I\right]^+.
\]
If $\bar y_r^a(m)>0$ is an interior minimizer, $\Xi_r^a(\cdot;m)$ is differentiable at $\bar y_r^a(m)$, and the right-hand side below lies in $[0,1]$, then
\begin{equation}
\label{eq:dynamic_critical_fractile}
F_r(\bar y_r^a(m)\mid m)
=
\frac{
b-c_r-\partial_y\Xi_r^a(\bar y_r^a(m);m)
}{
b+h
},
\end{equation}
where $\partial_y\Xi_r^a$ is the dynamic shadow value of capacity under continuation side $a$. At a kink point, the optimality condition is
\begin{equation}
\label{eq:dynamic_critical_fractile_subgradient}
0
\in
c_r
+
hF_r(\bar y_r^a(m)\mid m)
-
b\left[1-F_r(\bar y_r^a(m)\mid m)\right]
+
\partial\Xi_r^a(\bar y_r^a(m);m).
\end{equation}
If the global target is zero, or if the implemented capacity level reaches a feasible boundary, the corresponding one-sided or normal-cone optimality condition replaces \eqref{eq:dynamic_critical_fractile}.
\end{corollary}

The proof is provided in Appendix~\ref{app:base_capacity_representation}. The two side-specific targets determine $A_r^0$ and $A_r^1$, and the preservation--deployment rule in \eqref{eq:preservation_deployment_rule} selects which target to implement. Thus, the readiness comparison chooses the continuation side, while the base-capacity rule determines deployment conditional on that side.
\section{CBAP and Decision Guarantees}
\label{sec4}

This section develops a scenario-based implementation of the Bellman comparisons in Section~\ref{sec3}. Whereas the dynamic program uses the theoretical conditional belief $\mathbb P_r$, the decision maker uses a learned belief $\widehat{\mathbb P}_r$ constructed from cold-start information and post-launch observations. CBAP uses $\widehat{\mathbb P}_r$ to generate scenarios, estimates the close-after-current and preserve-after-current values, and applies the preservation--deployment decomposition.

\subsection{Scenario-Based Implementation}

At review $r$, CBAP takes as input the open-channel state $(\widehat{\mathbb P}_r,I_r,o_r=1)$, the cost parameters, and the feasible deployment set $\mathcal Q_r=[0,\bar q_r]$. The upper bound $\bar q_r$ represents an operational limit, such as a budget, cloud quota, GPU availability, or engineering capacity. If no hard limit binds, historical data or an independent pilot sample may determine the bound. Both $\bar q_r$ and the finite grid $\mathcal Q_r^G$, which contains zero, are fixed conditional on the current history before the online SAA scenarios are drawn.

\paragraph{Side-specific SAA value evaluation.} CBAP replaces each conditional expectation in $O_r^0$ and $O_r^1$ with a sample average approximation (SAA) over demand paths from $\widehat{\mathbb P}_r$. At review $r$, draw $D_{r:R}^{(n)}\sim\widehat{\mathbb P}_r$, $n=1,\ldots,N_r$. For any path-dependent cost functional $H(D_{r:R})$, approximate $\mathbb E_{\widehat{\mathbb P}_r}[H(D_{r:R})]$ by
\[
\frac{1}{N_r}
\sum_{n=1}^{N_r}
H(D_{r:R}^{(n)}).
\]
All candidate deployment quantities and continuation sides use the same scenario sample at a given review, allowing comparison under common demand realizations.

For candidate quantity $q\in\mathcal Q_r^G$ and continuation side $a\in\{0,1\}$, CBAP evaluates one-period operating cost and the side-specific continuation value on the common sample. The latent path determines realized cost, but the belief updater receives only information available to the operating policy. For each simulated candidate, it uses served demand $X_r(q)=\min\{D_r,I_r+q\}$ and the known capacity limit $I_r+q$. If $X_r(q)=I_r+q$, the update uses the weak event $D_r\ge I_r+q$ and does not distinguish equality from strict excess. The updater never receives latent excess demand. The next capacity state follows $g_r(I_r,q)$ rather than subtracting realized workload from capacity.

This observation kernel is an input to the belief-learning module rather than a contribution to censored-demand estimation. When potential workload is directly recorded, as in the BurstGPT application, the updater uses observed $D_r$ instead of capacity-limited $X_r$; the preservation--deployment decision rule is unchanged.

For $a\in\{0,1\}$, let $C_{r+1}^{a,(n)}(q)$ be the continuation estimate from the next belief-capacity state when $o_{r+1}=a$. We consider three evaluators. An exact learned-model Bellman evaluator solves the remaining grid-restricted dynamic program under $\widehat{\mathbb P}_r$ by exact integration or backward induction on a finite scenario tree. Each node action may depend only on the history observed at that node. A fitted evaluator uses
\begin{equation}
\label{eq:saa_continuation_term_by_side}
C_{r+1}^{a,(n)}(q)
=
\widetilde V_{r+1}^{a}
\left(
\widehat{\mathbb P}_{r+1}^{(n)}(q),
I_{r+1}(q);
\theta_{r+1}^{a}
\right),
\qquad a\in\{0,1\}.
\end{equation}
A rollout evaluator simulates a fixed non-anticipative continuation policy from $(\widehat{\mathbb P}_{r+1}^{(n)}(q),I_{r+1}(q),o_{r+1}=a)$. Future actions may depend on observations available up to the relevant epoch but not on the unrevealed part of the sampled path. A rollout estimates the value of its continuation policy rather than the optimal learned-model Bellman value. A nested SAA implementation introduces additional downstream sampling and optimization errors and is therefore also approximate. At the terminal boundary, $C_{R+1}^{a,(n)}(q)=0$ for $a\in\{0,1\}$.

The SAA counterparts of the close-after-current and preserve-after-current quantity values are
\begin{equation}
\label{eq:side_quantity_saa}
\widehat O_r^a(q;\widehat{\mathbb P}_r,I_r)
=
c_rq
+
\frac{1}{N_r}
\sum_{n=1}^{N_r}
\left[
hU_r^{(n)}(q)
+
bB_r^{(n)}(q)
+
C_{r+1}^{a,(n)}(q)
\right],
\qquad a\in\{0,1\}.
\end{equation}
The value $\widehat O_r^1$ excludes current readiness cost $k_r$, which is added separately to the preserve-after-current alternative. If the channel is closed at review $r$, the only feasible action is $(q_r,o_{r+1})=(0,0)$, with value estimate
\begin{equation}
\label{eq:closed_value_estimate_sec4}
\widehat V_r^0(\widehat{\mathbb P}_r,I_r)
=
\widehat O_r^0(0;\widehat{\mathbb P}_r,I_r).
\end{equation}

Algorithm~\ref{alg:value_estimation} evaluates $\widehat O_r^a(q;\widehat{\mathbb P}_r,I_r)$ for candidate pair $(q,a)$, where $a$ is the next channel state. Each pair uses the common sample drawn by Algorithm~\ref{alg:cbap}; no scenarios are resampled. Unlike a demand-depleting inventory recursion, the algorithm obtains the next capacity state from $g_r(I_r,q)$ and reuses it across scenarios for the same $q$. Service outcomes and belief updates remain scenario dependent.

\begin{algorithm}[htp]
\small
\caption{Side-Specific Value Evaluation}
\label{alg:value_estimation}
\begin{algorithmic}
\STATE \textbf{Input:} Open-channel state $(\widehat{\mathbb P}_r,I_r,o_r=1)$; candidate quantity $q\in\mathcal Q_r^G$; next channel state $a\in\{0,1\}$; sampled demand paths $\{D_{r:R}^{(n)}\}_{n=1}^{N_r}$; cost parameters from epoch $r$ onward; an exact learned-model Bellman evaluator, a fixed nonanticipative rollout evaluator, or a fitted continuation value $\widetilde V_{r+1}^{a}(\cdot;\theta_{r+1}^{a})$.
\STATE Compute next capacity state $I_{r+1}(q)=g_r(I_r,q)$.
\FOR{$n=1,\ldots,N_r$}
\STATE Compute $U_r^{(n)}(q)=\bigl[I_r+q-D_r^{(n)}\bigr]^+$ and $B_r^{(n)}(q)=\bigl[D_r^{(n)}-I_r-q\bigr]^+$.
\STATE Compute observable served demand $X_r^{(n)}(q)=\min\{D_r^{(n)},I_r+q\}$.
\STATE Update the belief using $X_r^{(n)}(q)$ and the known limit $I_r+q$; when $X_r^{(n)}(q)=I_r+q$, use the event $D_r^{(n)}\ge I_r+q$.
\IF{fitted continuation value is used}
\STATE Compute continuation term $C_{r+1}^{a,(n)}(q)$ via \eqref{eq:saa_continuation_term_by_side}.
\ELSIF{exact learned-model Bellman evaluation is used}
\STATE Compute $C_{r+1}^{a,(n)}(q)$ from the remaining grid-restricted Bellman recursion.
\ELSE
\STATE Compute $C_{r+1}^{a,(n)}(q)$ under the fixed nonanticipative rollout policy.
\ENDIF
\ENDFOR
\STATE Compute SAA counterpart $\widehat O_r^a(q;\widehat{\mathbb P}_r,I_r)$ via \eqref{eq:side_quantity_saa}.
\STATE \textbf{Return:} $\widehat O_r^a(q;\widehat{\mathbb P}_r,I_r)$.
\end{algorithmic}
\end{algorithm}

The corresponding continuation-evaluation errors are incorporated into $e_r^{\mathrm{cont}}$ in Section~\ref{sec:finite_sample_guarantees}.

\paragraph{Continuation-value approximation.}
When exact learned-model Bellman evaluation is computationally expensive, supervised function approximators $\widetilde V_r^a(\widehat{\mathbb P}_r,I_r;\theta_r^a)$, $a\in\{0,1\}$, can estimate the side-specific continuation values. Depending on the belief representation, the approximators may be neural networks, regression models, or other prediction architectures. They are trained backward over review epochs using Bellman targets from simulated or historical belief-capacity states; Appendix~\ref{app:fitted_continuation_training} provides the procedure. A mean-square training objective alone does not imply a uniform approximation bound. Any bound used below applies to a stated reachable set and requires separate justification or validation. The empirical application in Section~\ref{sec:empirical_application} uses scenario-based continuation evaluation; the fitted module is included as a scalable alternative for larger state spaces.

\paragraph{Online decision rule.}
At an open-channel state, CBAP evaluates the side-specific quantity values
$\widehat O_r^a(q;\widehat{\mathbb P}_r,I_r)$ for $a\in\{0,1\}$ and
$q\in\mathcal Q_r^G$ using the common SAA scenario sample. It then solves
\begin{equation}
\label{eq:qhat_by_side}
\widehat q_r^a
\in
\arg\min_{q\in\mathcal Q_r^G}
\widehat O_r^a(q;\widehat{\mathbb P}_r,I_r),
\qquad a\in\{0,1\},
\end{equation}
and defines
\begin{equation}
\label{eq:Ahat_by_side}
\widehat A_r^a(\widehat{\mathbb P}_r,I_r)
=
\widehat O_r^a(\widehat q_r^a;\widehat{\mathbb P}_r,I_r),
\qquad a\in\{0,1\}.
\end{equation}
The estimated gross preservation value is
\begin{equation}
\label{eq:Ghat_sec4}
\widehat G_r(\widehat{\mathbb P}_r,I_r)
=
\widehat A_r^0(\widehat{\mathbb P}_r,I_r)
-
\widehat A_r^1(\widehat{\mathbb P}_r,I_r),
\end{equation}
which is the sample counterpart of \eqref{eq:gross_option_value}. With optional closure buffer $\rho_r\ge0$, CBAP closes the channel only when $\widehat G_r(\widehat{\mathbb P}_r,I_r)\le k_r-\rho_r$. The buffer guards against premature closure when estimated gross preservation value is close to readiness cost. The baseline implementation sets $\rho_r=0$.

\begin{algorithm}[htp]
\small
\caption{Cold-Start Belief Actionability Policy}
\label{alg:cbap}
\begin{algorithmic}
\STATE \textbf{Input:} Open-channel state $(\widehat{\mathbb P}_r,I_r,o_r=1)$; costs $c_r,k_r,h,b$; deployment grid $\mathcal Q_r^G$; scenario size $N_r$; closure buffer $\rho_r\ge 0$; exact learned-model Bellman evaluator, fixed nonanticipative rollout evaluator, or fitted continuation values $\widetilde V_{r+1}^{a}(\cdot;\theta_{r+1}^{a})$, $a\in\{0,1\}$.
\STATE Draw sampled demand paths from $\widehat{\mathbb P}_r$.
\FOR{each $a\in\{0,1\}$ and each $q\in\mathcal Q_r^G$}
    \STATE Evaluate $\widehat O_r^a(q;\widehat{\mathbb P}_r,I_r)$ using Algorithm~\ref{alg:value_estimation} with the same sampled paths.
\ENDFOR
\STATE Compute $\widehat q_r^a$, $\widehat A_r^a$, $a\in\{0,1\}$, and $\widehat G_r(\widehat{\mathbb P}_r,I_r)$ using \eqref{eq:qhat_by_side}--\eqref{eq:Ghat_sec4}.
\IF{$\widehat G_r(\widehat{\mathbb P}_r,I_r)\le k_r-\rho_r$}
    \STATE Set $o_{r+1}=0$, $q_r=\widehat q_r^0$, and choose mode
    \[
    \begin{cases}
    \textbf{\emph{Exit}}, & \widehat q_r^0=0,\\
    \textbf{\emph{Terminal Response}}, & \widehat q_r^0>0.
    \end{cases}
    \]
\ELSE
    \STATE Set $o_{r+1}=1$, $q_r=\widehat q_r^1$, and choose mode
    \[
    \begin{cases}
    \textbf{\emph{Standby}}, & \widehat q_r^1=0,\\
    \textbf{\emph{Sustained Response}}, & \widehat q_r^1>0.
    \end{cases}
    \]
\ENDIF
\STATE Observe the realized service outcome, update $I_{r+1}=g_r(I_r,q_r)$ and $\widehat{\mathbb P}_{r+1}$, and proceed to the next review epoch.
\end{algorithmic}
\end{algorithm}

Appendix~\ref{app:fitted_continuation_training} provides the optional fitted-continuation training procedure.
\subsection{Local Decision Guarantees and Lifecycle Error Bound}
\label{sec:finite_sample_guarantees}

We distinguish a local decision guarantee from an expected lifecycle guarantee because CBAP draws new scenarios only at histories reached by the online policy. A high-probability error bound along one realized trajectory does not imply the same bound for the expected value of the full randomized policy.

Let $\mathcal H_r$ denote the information available immediately before the SAA scenarios are drawn at epoch $r$, including the learned belief, the deployment grid, and any continuation evaluator fixed before online sampling. The deployment grid is therefore $\mathcal H_r$-measurable. At an open-channel state, define $\mathcal K_r=\mathcal Q_r^G\times\{0,1\}$. For $u=(q,a)\in\mathcal K_r$, let $Q_r^{\mathbb P}(u)$ be the true grid-restricted Bellman candidate value, including $k_ra$, and let $Q_r^{\mathrm{learn}}(u)$ be the population candidate value under the learned predictive model with exact grid-restricted Bellman continuation. Let $\widehat Q_r(u)$ be the estimate used by CBAP, where $\widehat Q_r((q,a))=\widehat O_r^a(q;\widehat{\mathbb P}_r,I_r)+k_ra$. Suppressing dependence on the current state, let
\[
\bar Q_r(u)
=
\mathbb E
\left[
\widehat Q_r(u)
\mid
\mathcal H_r
\right].
\]
The continuation-evaluation error below is zero for an exact learned-model Bellman evaluator. For fitted values, it includes the difference between $\widetilde V_{r+1}^{a}$ and the exact learned-model continuation value over the next states used in candidate evaluation. For a fixed rollout policy, it also includes policy-evaluation bias relative to the optimal learned-model value. For nested SAA, it further includes downstream sampling and optimization errors.

\begin{assumption}[Candidate-value accuracy and adaptive sampling]
\label{ass:belief_stability}
For each epoch $r$, the following conditions hold.
\begin{enumerate}
\item The belief-accuracy event
\[
\mathcal E_r^{\mathrm{bel}}
=
\left\{
\sup_{u\in\mathcal K_r}
\left|
Q_r^{\mathbb P}(u)-Q_r^{\mathrm{learn}}(u)
\right|
\le
e_r^{\mathrm{bel}}
\right\}
\]
has probability at least $1-\beta_r$ under the history distribution generated by CBAP.

\item The continuation-evaluation event
\[
\mathcal E_r^{\mathrm{cont}}
=
\left\{
\sup_{u\in\mathcal K_r}
\left|
\bar Q_r(u)-Q_r^{\mathrm{learn}}(u)
\right|
\le
e_r^{\mathrm{cont}}
\right\}
\]
has probability at least $1-\varphi_r$ under the joint distribution of the continuation evaluator and the histories generated by CBAP. If the evaluator is fixed before online evaluation and the bound is deterministic, set $\varphi_r=0$.

\item Conditional on $\mathcal H_r$, the sampling event
\[
\mathcal E_r^{\mathrm{saa}}
=
\left\{
\sup_{u\in\mathcal K_r}
\left|
\widehat Q_r(u)-\bar Q_r(u)
\right|
\le
e_r^{\mathrm{saa}}
\right\}
\]
satisfies $\Pr((\mathcal E_r^{\mathrm{saa}})^c\mid\mathcal H_r)\le\gamma_r$.

\item The true one-step Bellman gap is bounded on the reachable set: for every $u\in\mathcal K_r$,
\[
0
\le
Q_r^{\mathbb P}(u)
-
\min_{v\in\mathcal K_r}Q_r^{\mathbb P}(v)
\le
\overline\Delta_r.
\]
\end{enumerate}
\end{assumption}

Assumption~\ref{ass:belief_stability} uses direct candidate-value error rather than a metric on the path-belief space. A Wasserstein bound is a sufficient condition if the path metric is specified and the Bellman candidates are Lipschitz under that metric. The result concerns error propagation, not the learning rate of a particular belief estimator.

The sampling radius $e_r^{\mathrm{saa}}$ must match the scenario distribution and estimator. If, for example, the one-scenario evaluation for every candidate lies in an interval of length $L_r^{\mathrm{eval}}$, conditional Hoeffding concentration and a union bound give
\begin{equation}
\label{eq:saa_radius_bounded}
e_r^{\mathrm{saa}}
=
L_r^{\mathrm{eval}}
\sqrt{
\frac{\log(2|\mathcal K_r|/\gamma_r)}{2N_r}
}.
\end{equation}
This bounded-range formula does not apply directly to unbounded demand. In that case, $e_r^{\mathrm{saa}}$ must follow from a suitable tail condition, a robust estimator, or an explicitly truncated operational model. Averaging all paths in a finite empirical scenario law creates no SAA error relative to that law.

Let $\eta_{\mathrm{opt},r}$ denote the error from minimization over the fixed finite grid. It is zero under exact enumeration. Because the comparator below uses the same grid, $\eta_{\mathrm{opt},r}$ excludes grid-approximation error.

\begin{theorem}[Local and lifecycle error bounds for scenario-based CBAP]
\label{thm:cbap_regret}
Suppose Assumption~\ref{ass:belief_stability} holds. Define
\begin{equation}
\label{eq:local_cbap_error}
E_r
=
2
\left(
e_r^{\mathrm{bel}}
+
e_r^{\mathrm{saa}}
+
e_r^{\mathrm{cont}}
\right)
+
\eta_{\mathrm{opt},r}
+
\rho_r.
\end{equation}
At any fixed reachable open-channel history, on $\mathcal E_r^{\mathrm{bel}}\cap\mathcal E_r^{\mathrm{saa}}\cap\mathcal E_r^{\mathrm{cont}}$, the action $\widehat u_r$ selected by CBAP satisfies
\begin{equation}
\label{eq:cbap_local_bound}
Q_r^{\mathbb P}(\widehat u_r)
-
\min_{u\in\mathcal K_r}Q_r^{\mathbb P}(u)
\le
E_r.
\end{equation}

Let $V_{\mathbb P}^{\mathrm{CBAP}}(\mathcal S_1)$ denote lifecycle cost averaged over both demand and the internal scenario draws of the online algorithm. Let $\pi_G^*$ be the true-model optimal policy on the same grids. Then
\begin{equation}
\label{eq:cbap_regret_bound}
V_{\mathbb P}^{\mathrm{CBAP}}(\mathcal S_1)
-
V_{\mathbb P}^{\pi_G^*}(\mathcal S_1)
\le
\sum_{r=1}^{R}
\left[
E_r
+
(\beta_r+\gamma_r+\varphi_r)\overline\Delta_r
\right].
\end{equation}
\end{theorem}

\noindent
The proof is provided in Appendix~\ref{app:proof_cbap_regret}.

Theorem~\ref{thm:cbap_regret} does not condition sampling probability on the belief-accuracy event. It controls $\mathcal E_r^{\mathrm{bel}}\cap\mathcal E_r^{\mathrm{saa}}\cap\mathcal E_r^{\mathrm{cont}}$ and retains every failure penalty in the expected lifecycle bound. The local bound applies at a fixed history, while the lifecycle bound averages true Bellman gaps over the histories reached by the randomized online policy.

For fixed $R$, convergence to the grid-restricted optimum follows if the belief, sampling, continuation, optimization, and buffering errors vanish and $\beta_r,\gamma_r,\varphi_r\to0$. Under \eqref{eq:saa_radius_bounded}, convergence also requires $\log(|\mathcal K_r|/\gamma_r)/N_r\to0$. Comparison with the continuous-action optimum requires a separate grid-approximation bound that vanishes as the grids become dense. Appendix~\ref{app:mode_consistency_note} gives the corresponding local mode-stability result.
\section{Numerical Experiments and BurstGPT Application}
\label{sec5}

This section examines the operational and economic performance of CBAP. The controlled experiment isolates the value of post-launch response and selective readiness preservation. The empirical application uses BurstGPT workload traces and a daily capacity-readiness review. This daily review represents managerial capacity planning rather than lower-level autoscaling: hourly token workloads are aggregated into daily demand, and capacity-readiness decisions are revised once per day. In both studies, capacity follows $I_{r+1}=(1-\delta)(I_r+q_r)$. Workload determines service, idle capacity, shortage, and belief updating, but does not deplete capacity. We set $\delta=0.25$, so deployed capacity remains available across review intervals but gradually decays over the lifecycle.

\subsection{Benchmark Policies and Evaluation Metrics}
\label{sec:evaluation_protocol}

\paragraph{Benchmark policies.}
We compare CBAP with six benchmarks. Table~\ref{tab:benchmark_definitions} summarizes each comparison, and the policy rules are defined below.

Static chooses only pre-launch deployment $q_0$ and allows no post-launch response. It represents launch-time capacity sizing without response flexibility. Forecast-then-deploy updates the demand belief after launch and converts the forecast into a target capacity level but does not value whether the response channel should remain open. It follows the forecast-then-optimize logic used in data-driven inventory and capacity decisions \citep{BanGallienMersereau2019}. To avoid front-loading capacity for distant peaks under the decaying transition, the policy uses a short lookahead window: two review periods in the controlled experiment and two operating days in the BurstGPT application.

Always-open is an externally imposed full-window commitment policy: it preserves the response channel at every scheduled review, pays the associated readiness cost, and optimizes deployment subject to that constraint. The comparison with CBAP therefore measures the cost of full-window readiness commitment relative to endogenous closure. Myopic response also keeps the channel open but chooses $q$ to minimize current-period expected deployment and operating cost,
$c_rq+\mathbb E_{\widehat{\mathbb P}_r}[\ell_r(I_r,q,D_r)]$,
without continuation value. The two benchmarks represent flexibility preserved by design and response based only on current-period cost, respectively
\citep{EppenIyer1997,SerelDadaMoskowitz2001,HazraMahadevan2009,HuhRusmevichientong2009,BesbesMuharremoglu2013}.

Two-stage recourse chooses $q_0$ before launch and permits one post-launch deployment at a fixed demand-update point, consistent with demand-update sourcing models \citep{FedergruenLiuLu2026}. It preserves the channel and pays readiness cost before that point; at the recourse epoch, it may deploy once and then closes the channel. The controlled experiment places recourse at the middle review, whereas the BurstGPT application selects the recourse day using the timing diagnostic in Appendix~\ref{app:two_stage_timing}. Oracle observes the realized future demand path and optimizes ex post. It is not implementable and serves only as a diagnostic lower bound.

\begin{table}[htbp]
\centering
\small
\caption{Benchmark policies and comparison purposes}
\label{tab:benchmark_definitions}
\renewcommand{\arraystretch}{1.15}
\begin{tabularx}{\textwidth}{
>{\raggedright\arraybackslash}p{0.25\textwidth}
>{\raggedright\arraybackslash}X}
\toprule
\textbf{Policy} & \textbf{Purpose} \\
\midrule
Static & Tests the value of post-launch response. \\[3pt]
Forecast-then-deploy & Tests whether belief updating is sufficient without valuing response readiness. \\[3pt]
Always-open & Tests the cost of a full-window readiness commitment relative to endogenous closure. \\[3pt]
Myopic response & Tests the value of dynamic continuation and future learning. \\[3pt]
Two-stage recourse & Tests the value of repeated, state-dependent response timing. \\[3pt]
Oracle & Provides a non-implementable diagnostic lower bound. \\
\bottomrule
\end{tabularx}
\end{table}
\paragraph{Evaluation protocol and metrics.}
We use a rolling lifecycle evaluation. For each launch instance, an implementable policy chooses pre-launch capacity, observes only information available by each review, updates its demand belief, and makes a capacity-readiness decision. Future demand is hidden from implementable policies and used only for out-of-sample evaluation. Realized lifecycle cost follows the objective in Section~\ref{sec2}, with readiness cost incurred only when the channel is preserved.

For launch instance $i$, let $\mathrm{Cost}_i^\pi$ denote the realized lifecycle cost under policy $\pi$. Average lifecycle cost is the primary economic objective:
\begin{equation}
\overline{\mathrm{Cost}}^\pi=
\frac{1}{M}
\sum_{i=1}^{M}
\mathrm{Cost}_i^\pi,
\label{eq:avg_lifecycle_cost}
\end{equation}
where $M$ is the number of out-of-sample launch instances.

Fill rate measures the fraction of realized demand that is served:
\begin{equation}
\mathrm{FillRate}^{\pi}
=
\frac{
\sum_{i=1}^{M}
\sum_{s=1}^{T_i}
X_{is}^{\pi}
}{
\sum_{i=1}^{M}
\sum_{s=1}^{T_i}
D_{is}
},
\label{eq:fill_rate_experiment}
\end{equation}
where $X_{is}^{\pi}$ is served demand for instance $i$ in lifecycle period $s$. Fill rate is a service diagnostic rather than the objective. Deploying more capacity or preserving readiness more often may increase fill rate but can also raise deployment, idle-capacity, or readiness costs.

Readiness cost is the average expenditure on preserving the response channel. Response actions count positive post-launch deployments per lifecycle, excluding standby decisions. Together, these measures distinguish the cost of preserving readiness from its use through deployment.

Normalized regret measures how much of the Static-to-Oracle performance gap remains:
\begin{equation}
\mathrm{NormalizedRegret}^{\pi}
=
\frac{
\overline{\mathrm{Cost}}^{\pi}
-
\overline{\mathrm{Cost}}^{\mathrm{Oracle}}
}{
\overline{\mathrm{Cost}}^{\mathrm{Static}}
-
\overline{\mathrm{Cost}}^{\mathrm{Oracle}}
}.
\label{eq:normalized_regret_experiment}
\end{equation}
Oracle has normalized regret zero and Static has normalized regret one. This metric puts results from the controlled simulation and empirical application on the same relative scale.

Mode shares report the fraction of review states assigned to Exit, Standby, Terminal Response, and Sustained Response. We also use three aggregate measures: closing is Exit plus Terminal Response, preserving is Standby plus Sustained Response, and positive response is Terminal Response plus Sustained Response.
\paragraph{Reproducibility and randomization.}
All policies are evaluated on the same controlled demand realizations and the same fixed set of empirical episodes. Within each CBAP review, all candidate quantities and continuation sides are evaluated using common scenario draws, as described in Section~4.1. The BurstGPT test sample is selected once using the reported seed 202604 and is then held fixed across policies.
\subsection{Controlled Numerical Experiment}
\label{sec:numerical_experiment}

\paragraph{Simulation design.}
The controlled experiment isolates the main economic mechanisms. We simulate 12-period product lifecycles with post-launch reviews in periods $2,\ldots,11$. Each product belongs to an unobserved Low, Regular, or Burst demand regime, representing cold-start uncertainty about weak adoption, regular adoption, or a post-launch burst. The prior probabilities are $0.35$, $0.45$, and $0.20$, respectively.

Demand for product $i$ in period $s$ follows $D_{is}\mid\theta_i,\xi_i\sim\mathrm{Poisson}(\xi_i\lambda_{\theta_i,s})$, where $\xi_i$ is a product-level random effect with $\log\xi_i\sim N(0,\sigma_\xi^2)$. The three regimes have different lifecycle profiles and increasing peak intensity, so early observations are informative but do not immediately reveal the regime. Demand is observed through serviceable capacity. The baseline uses a Bayesian filter over the three regimes and accounts for the known capacity limit when updating the belief. We use a simple belief generator to keep the experiment focused on the decision layer rather than forecasting performance.

The capacity transition is $I_{r+1}=(1-\delta)(I_r+q_r)$ with $\delta=0.25$, so capacity persists across periods but may expire, be released, or be scaled in after each review interval. Demand affects service, idle capacity, shortage, and belief updating but does not deplete capacity. Pre-launch deployment cost is $c_0=2.0$. Post-launch deployment cost starts at $3.0$ and rises by $0.10$ per review index, representing a later-response urgency premium \citep{WangAtasuKurtulus2012,FedergruenLiuLu2026}. We set idle-capacity cost $h=0.5$, shortage or latency penalty $b=10$, and readiness cost $k_r=40$. Deployment uses a capacity grid with step size 10. The experiment includes 300 training lifecycles, 120 out-of-sample lifecycles, and 8 SAA scenarios per review decision. Forecast-then-deploy uses a two-period lookahead, and Two-stage recourse uses the middle review.

\paragraph{Main numerical results.}
Table~\ref{tab:baseline_results} reports the baseline results. CBAP has the lowest lifecycle cost among implementable policies and leaves 4.7\% of the Static-to-Oracle cost gap. Its cost is 65.18\% lower than Static, 16.20\% lower than Forecast-then-deploy, 9.65\% lower than Myopic response, and 7.90\% lower than Always-open. These comparisons provide mechanism-based evidence rather than an exact causal decomposition. The large gap relative to Static is consistent with the benefit of post-launch response. The comparison with Forecast-then-deploy suggests that updated forecasts alone are insufficient when preserving response readiness is costly. The comparisons with Myopic response and Always-open indicate additional gains associated with continuation-aware decisions and selective closure. Forecast-then-deploy attains a higher fill rate than CBAP, 0.9916 versus 0.9721, but also has a higher total cost. The cost objective therefore does not maximize service without regard to deployment, idle-capacity, and readiness costs.

\begin{table}[htbp]
\centering
\small
\caption{Baseline numerical policy comparison}
\label{tab:baseline_results}
\renewcommand{\arraystretch}{1.15}
\setlength{\tabcolsep}{3.5pt}
\begin{tabularx}{\textwidth}{
>{\raggedright\arraybackslash}p{0.22\textwidth}
>{\centering\arraybackslash}p{0.19\textwidth}
>{\centering\arraybackslash}p{0.11\textwidth}
>{\centering\arraybackslash}p{0.13\textwidth}
>{\centering\arraybackslash}p{0.15\textwidth}
>{\centering\arraybackslash}X}
\toprule
\textbf{Policy}
& \textbf{Avg. cost}
& \textbf{Fill rate}
& \textbf{Ready. cost}
& \textbf{Resp. actions}
& \textbf{Norm. regret} \\
\midrule
\multicolumn{6}{l}{\textit{Non-implementable diagnostic bound}} \\
Oracle & 2074.76 & 0.9833 & 206.67 & 4.11 & 0.000 \\
\midrule
\multicolumn{6}{l}{\textit{Implementable policies}} \\
\textbf{CBAP} & \textbf{2285.89} & 0.9721 & 123.33 & 3.63 & \textbf{0.047} \\
Always-open & 2482.02 {\scriptsize (+7.90\%)} & 0.9807 & 400.00 & 4.83 & 0.091 \\
Forecast-then-deploy & 2727.70 {\scriptsize (+16.20\%)} & 0.9916 & 400.00 & 4.45 & 0.145 \\
Myopic response & 2530.06 {\scriptsize (+9.65\%)} & 0.9705 & 400.00 & 5.33 & 0.101 \\
Two-stage recourse & 6421.57 {\scriptsize (+64.40\%)} & 0.7775 & 160.00 & 1.00 & 0.968 \\
Static & 6565.56 {\scriptsize (+65.18\%)} & 0.8886 & 0.00 & 0.00 & 1.000 \\
\bottomrule
\end{tabularx}
\vspace{3pt}
\begin{minipage}{0.97\textwidth}
\footnotesize
\emph{Notes.} Percentages in parentheses report the cost reduction achieved by CBAP relative to the corresponding benchmark. Oracle observes future demand and is not implementable. ``Resp. actions'' is the average number of post-launch review epochs with positive deployment. The capacity transition is $I_{r+1}=(1-\delta)(I_r+q_r)$ with $\delta=0.25$.
\end{minipage}
\end{table}

Figure~\ref{fig:main_numerical_mechanism} explains these cost differences. CBAP uses all four modes, with Exit and Sustained Response occurring most often in this calibration. The cost decomposition shows that CBAP avoids the high shortage cost of Static while spending much less on readiness than policies that keep the channel open. As $k_r$ rises, CBAP uses preservation modes less often, which is consistent with the readiness comparison in Section~\ref{sec3}.

\begin{figure}[htbp]
\centering
\subfigure[Four-action mode shares across review states]{
\includegraphics[width=3in]{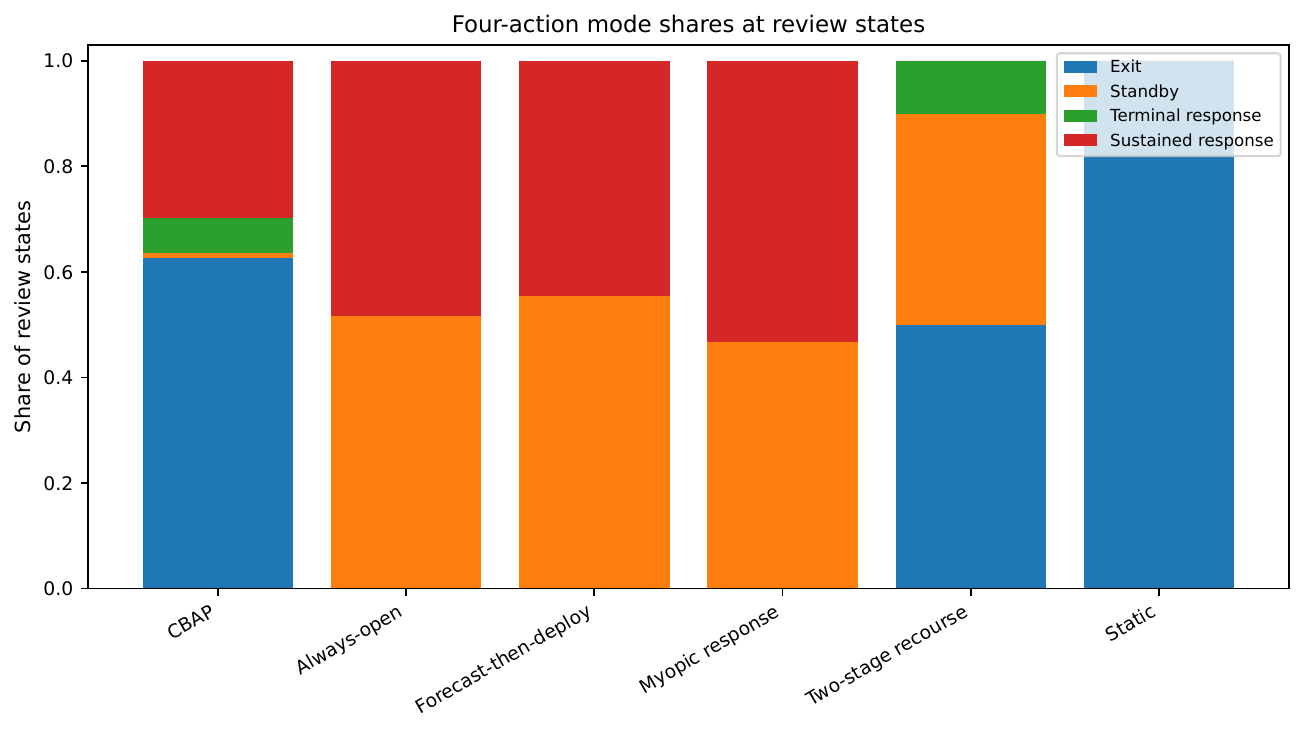}
\label{fig:mode_shares_four_action}
}
\subfigure[Cost decomposition by policy]{
\includegraphics[width=3in]{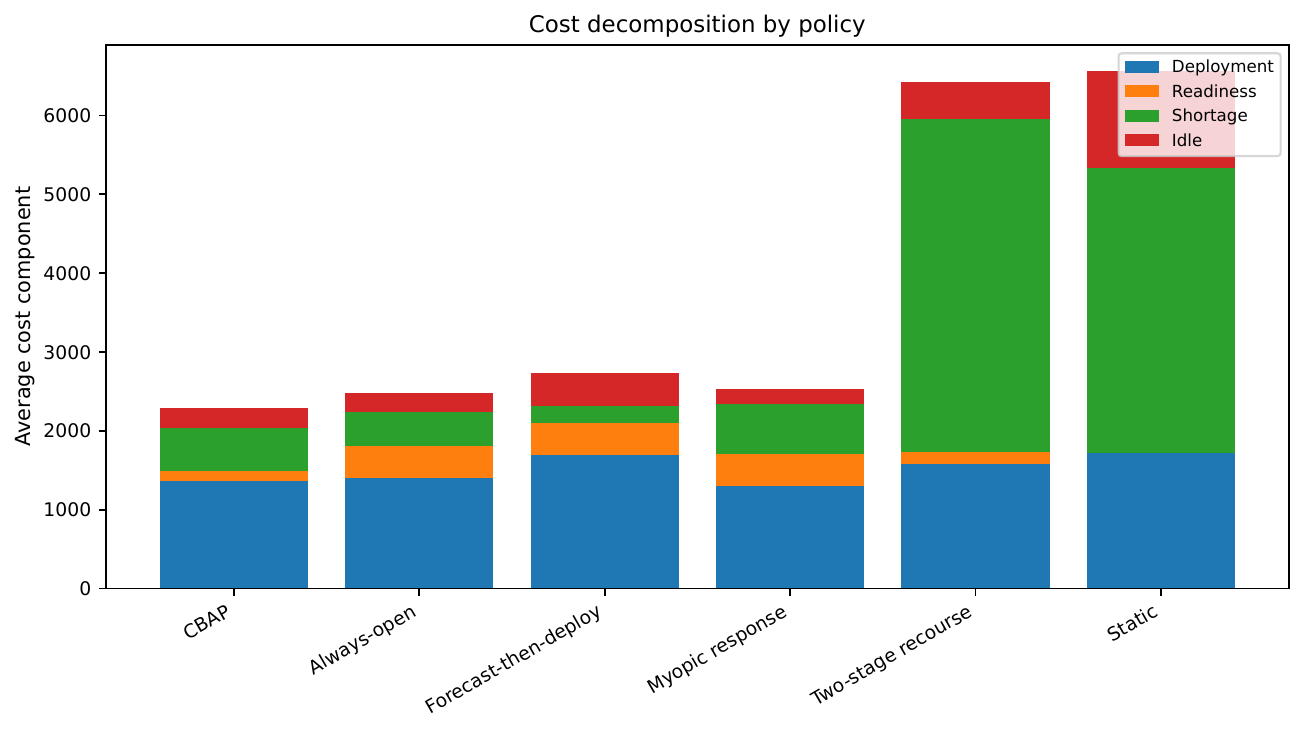}
\label{fig:cost_decomposition}
}
\subfigure[Sensitivity to readiness cost $k_r$]{
\includegraphics[width=3in]{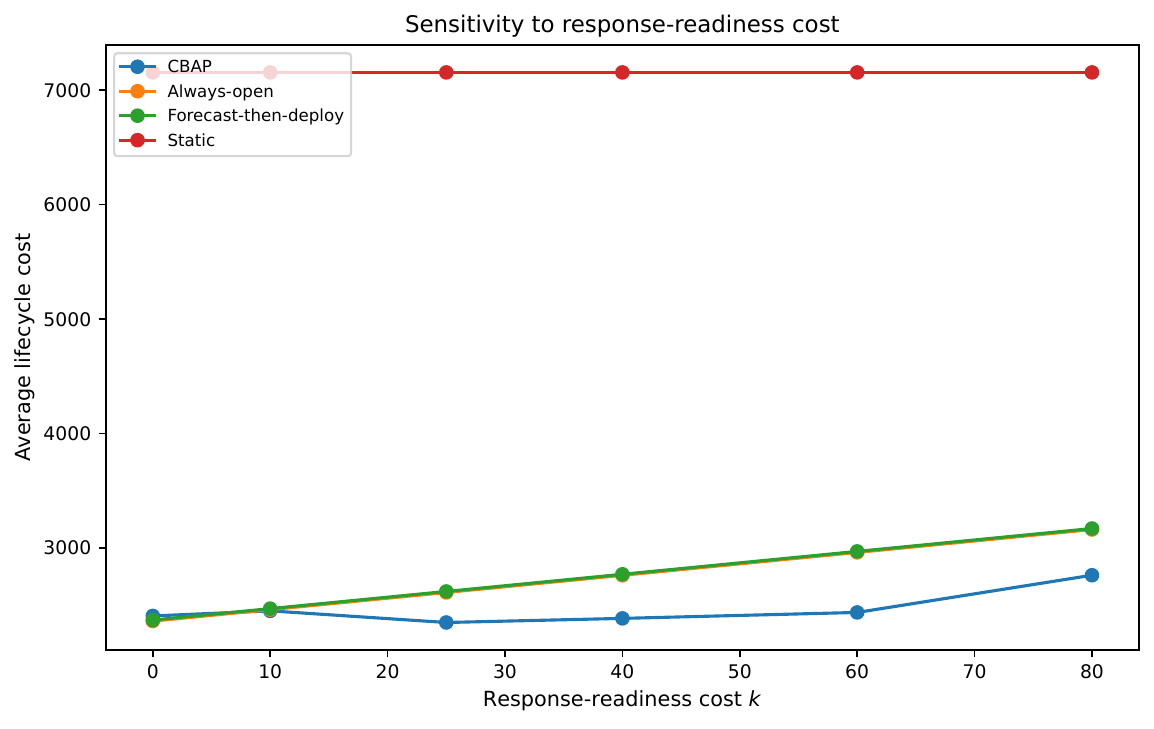}
\label{fig:k_sensitivity}
}
\caption{Numerical mechanisms: mode selection, cost decomposition, and readiness-cost sensitivity.}
\label{fig:main_numerical_mechanism}
\end{figure}

Two-stage recourse can respond only once, at the middle review epoch, and then closes the channel. With $\delta=0.25$, capacity deployed at this fixed point decays over the remaining lifecycle and cannot accommodate late burst paths well. The comparison therefore concerns the limitation of a fixed recourse time in this setting; it does not imply that one-shot recourse is dominated in every capacity-planning problem.

The additional diagnostics in Appendix~\ref{app:additional_diagnostics} and the regime analysis in Appendix~\ref{app:regime_analysis} show where these gains arise. CBAP closes the channel frequently on low-demand paths and preserves it more often when demand is regular or bursty. It therefore combines post-launch response with readiness spending that depends on the observed belief path.

\subsection{BurstGPT Application: Daily Capacity-Readiness Review}
\label{sec:empirical_application}

\paragraph{Data and empirical setting.}
We evaluate CBAP on BurstGPT, a public LLM-serving workload trace from Azure OpenAI GPT services \citep{BurstGPT2025}. Token workload is treated as observed demand, and decisions are made at the daily capacity-planning level. Daily workload retains hourly variation, but CBAP decides only once per day whether to deploy capacity and preserve the rapid-response channel. The application therefore differs from second- or minute-level autoscaling.

The main empirical sample is \texttt{BurstGPT\_1.csv}, which contains 1,429,737 request records over approximately 61 days. Each record includes a timestamp, model type, request-token length, response-token length, total-token length, and log type. We define workload group $g=(\mathrm{ModelType},\mathrm{LogType})$ as a model--log-type pair, producing four groups in the main sample. Requests within each group are aggregated into hourly token workload. For group $g$ and hour $t$, 
$X_{gt}=\sum_{j\in\mathcal R_{gt}} \mathrm{TotalTokens}_{j}$, where $\mathcal R_{gt}$ is the set of requests in group $g$ during hour $t$. We convert token workload to capacity-hour units using $\widetilde X_{gt}=X_{gt}/\mu$, where $\mu=52{,}707.5$ tokens per capacity-hour is the median positive hourly token workload in the sample. This normalization retains the burstiness and temporal concentration of the trace while expressing demand, deployment, shortage, and idle capacity in common units.

Each empirical episode is a rolling seven-day lifecycle within one workload group. An episode is therefore a contiguous segment from a single model--log-type path rather than a segment pooled across groups. Let $g_i$ be the group of episode $i$, and let $t_i$ be its first hour.Daily capacity-hour demand on operating day $d$ is $D_{id}
=
\sum_{\tau=1}^{24}
\widetilde X_{g_i,t_i+24(d-1)+\tau}$, $ d=1,\ldots,7$. Initial capacity is selected before day 1. Six post-launch reviews occur at the beginning of days 1--6; at each review, a policy uses the observed workload prefix through the preceding day to construct a scenario belief over the remaining path and select $(q_r,o_{r+1})$. Hourly profiles form the daily realizations and retain the empirical burst pattern, while response-readiness decisions remain at the daily planning level.

BurstGPT records incoming request and token workload, which we treat as directly observed potential workload. Historical-analog beliefs are therefore constructed from observed workload prefixes rather than capacity-truncated service observations. This application validates the response-readiness decision layer under a common predictive belief; it does not estimate or validate a censored-demand learning model. All implementable policies use the same historical-analog beliefs.

The rolling construction produces 5,092 seven-day windows. To avoid look-ahead bias, we use 3,562 earlier windows as the training pool for historical-analog scenario construction and benchmark calibration, and reserve 1,530 later windows for out-of-sample testing. From these 1,530 held-out windows, we select 500 without replacement using proportional stratified sampling across the four workload groups, with largest-remainder allocation and random seed 202604. The resulting episode-index list is fixed before policy evaluation and is used for every policy, yielding 3,000 daily review states. 
For each selected test episode $i$, historical analogs are drawn only from the training pool and are additionally required to end strictly before the beginning of episode $i$. This episode-specific temporal embargo prevents an overlapping training window from containing observations from the evaluated lifecycle. At each review, eligible analogs are restricted to the same workload group and ranked by similarity in the observed daily prefix.

\paragraph{Cost calibration and belief construction.}
The empirical setting uses normalized cloud-price parameters. We interpret $c_r$ as the cost of deploying compute capacity at review $r$. Pay-as-you-go cloud pricing motivates the cost structure, but the parameters are normalized rather than expressed in dollars. Let $\bar c$ be the cost of one capacity unit over a daily review interval. We set $\bar c=3$, $c_0=3$, and 
$c_r=\bar c\left(1+\lambda\frac{r}{L}\right)$, where $L=7$ is the number of operating days in the lifecycle and $\lambda=0.25$ is a later-response urgency premium. For the six daily review epochs, this gives deployment costs $3.107$, $3.214$, $3.321$, $3.429$, $3.536$, and $3.643$.

We interpret $k_r$ as the cost of preserving rapid response. EC2 On-Demand Capacity Reservations motivate this calibration because they reserve capacity in a specified Availability Zone and charge the equivalent On-Demand rate whether or not the capacity is used.\footnote{See Amazon Web Services, ``On-Demand Capacity Reservations,'' \url{https://docs.aws.amazon.com/AWSEC2/latest/UserGuide/ec2-capacity-reservations.html}, and ``Capacity Reservation pricing and billing,'' \url{https://docs.aws.amazon.com/AWSEC2/latest/UserGuide/capacity-reservations-pricing-billing.html}.} In the baseline, response readiness is equivalent to 16 standby capacity units with readiness intensity $\alpha=1$, so $k_r=\alpha\bar c\times16=48$. The idle-capacity cost is $h=0.5$, the shortage or latency penalty is $b=10$, and the capacity-release parameter is $\delta=0.25$.

All implementable policies use the same historical-analog scenario generator. At each daily review, the algorithm identifies training windows from the same model--log-type group with similar observed prefixes. The remaining paths of the 30 nearest analogs form the scenario belief for CBAP and every benchmark; all 30 paths are used in the candidate-value averages without resampling. Conditional on the selected test episode and its observed prefix, the empirical scenario belief is therefore deterministic. Ties in analog distance are resolved first by the earlier window start time and then by the episode index. Forecast-then-deploy uses a two-day lookahead. For Two-stage recourse, the candidate day with the lowest training-sample timing cost is day 4, which is fixed before out-of-sample evaluation.

\paragraph{Main empirical results.}
Table~\ref{tab:burstgpt_main_results} reports the out-of-sample results. CBAP has the lowest lifecycle cost among implementable policies and leaves 53.6\% of the Static-to-Oracle cost gap. Its cost is 33.44\% lower than Static, 19.40\% lower than Forecast-then-deploy, 18.70\% lower than Two-stage recourse, 3.20\% lower than Always-open, and 2.39\% lower than Myopic response. Always-open and Myopic response are the closest benchmarks. Their performance suggests that repeated daily response flexibility accounts for an important part of the cost reduction. The remaining differences are consistent with additional benefits from continuation-aware decisions and selective closure.

\begin{table}[htbp]
\centering
\small
\caption{BurstGPT daily-review policy comparison: out-of-sample evaluation}
\label{tab:burstgpt_main_results}
\renewcommand{\arraystretch}{1.15}
\setlength{\tabcolsep}{3.5pt}
\begin{tabularx}{\textwidth}{
>{\raggedright\arraybackslash}p{0.22\textwidth}
>{\centering\arraybackslash}p{0.19\textwidth}
>{\centering\arraybackslash}p{0.11\textwidth}
>{\centering\arraybackslash}p{0.13\textwidth}
>{\centering\arraybackslash}p{0.15\textwidth}
>{\centering\arraybackslash}X}
\toprule
\textbf{Policy}
& \textbf{Avg. cost}
& \textbf{Fill rate}
& \textbf{Ready. cost}
& \textbf{Resp. actions}
& \textbf{Norm. regret} \\
\midrule
\multicolumn{6}{l}{\textit{Non-implementable diagnostic bound}} \\
Oracle
& 1581.13
& 1.0000
& 151.10
& 3.148
& 0.000 \\
\midrule
\multicolumn{6}{l}{\textit{Implementable policies}} \\
\textbf{CBAP}
& \textbf{3774.92}
& 0.9106
& 165.31
& 1.546
& \textbf{0.536} \\
Always-open
& 3899.79 {\scriptsize (+3.20\%)}
& 0.9098
& 288.00
& 2.358
& 0.567 \\
Forecast-then-deploy
& 4683.36 {\scriptsize (+19.40\%)}
& 0.9478
& 288.00
& 2.574
& 0.758 \\
Myopic response
& 3867.42 {\scriptsize (+2.39\%)}
& 0.8851
& 288.00
& 2.072
& 0.559 \\
Two-stage recourse
& 4643.37 {\scriptsize (+18.70\%)}
& 0.9738
& 144.00
& 0.656
& 0.749 \\
Static
& 5671.56 {\scriptsize (+33.44\%)}
& 0.9878
& 0.00
& 0.000
& 1.000 \\
\bottomrule
\end{tabularx}
\vspace{3pt}
\begin{minipage}{0.97\textwidth}
\footnotesize
\emph{Notes.} Percentages in parentheses report the cost reduction achieved by CBAP relative to the corresponding benchmark. Oracle observes the realized future workload and is not implementable. ``Resp. actions'' is the average number of positive post-launch deployments per lifecycle. The daily-review application uses 500 out-of-sample rolling seven-day episodes.
\end{minipage}
\end{table}
Fill rate must be considered together with cost. Static and Forecast-then-deploy reach fill rates of 0.9878 and 0.9478, but their greater deployment and idle-capacity exposure raises total cost. Always-open has nearly the same fill rate as CBAP, 0.9098 versus 0.9106, but spends 288.00 on readiness rather than 165.31. This pattern is consistent with cost savings from selective closure under the common information set. Myopic response also spends 288.00 on readiness and has a lower fill rate of 0.8851, suggesting that decisions based only on current-period cost do not fully capture continuation value. Two-stage recourse achieves a high fill rate but cannot adapt its single response time to individual workload paths.

Figure~\ref{fig:burstgpt_daily_main} reports the corresponding cost and action patterns. CBAP limits the idle-capacity exposure of Static and Forecast-then-deploy and spends less on readiness than policies that preserve the channel by construction. It closes the channel in 42.6\% of review states and preserves it in 57.4\%. The mode shares are 35.4\% Exit, 38.8\% Standby, 7.2\% Terminal Response, and 18.6\% Sustained Response. All four modes occur, so both zero and positive deployment may be paired with either closure or preservation.

\begin{figure}[htbp]
\centering
\subfigure[Average lifecycle cost]{
\includegraphics[width=3in]{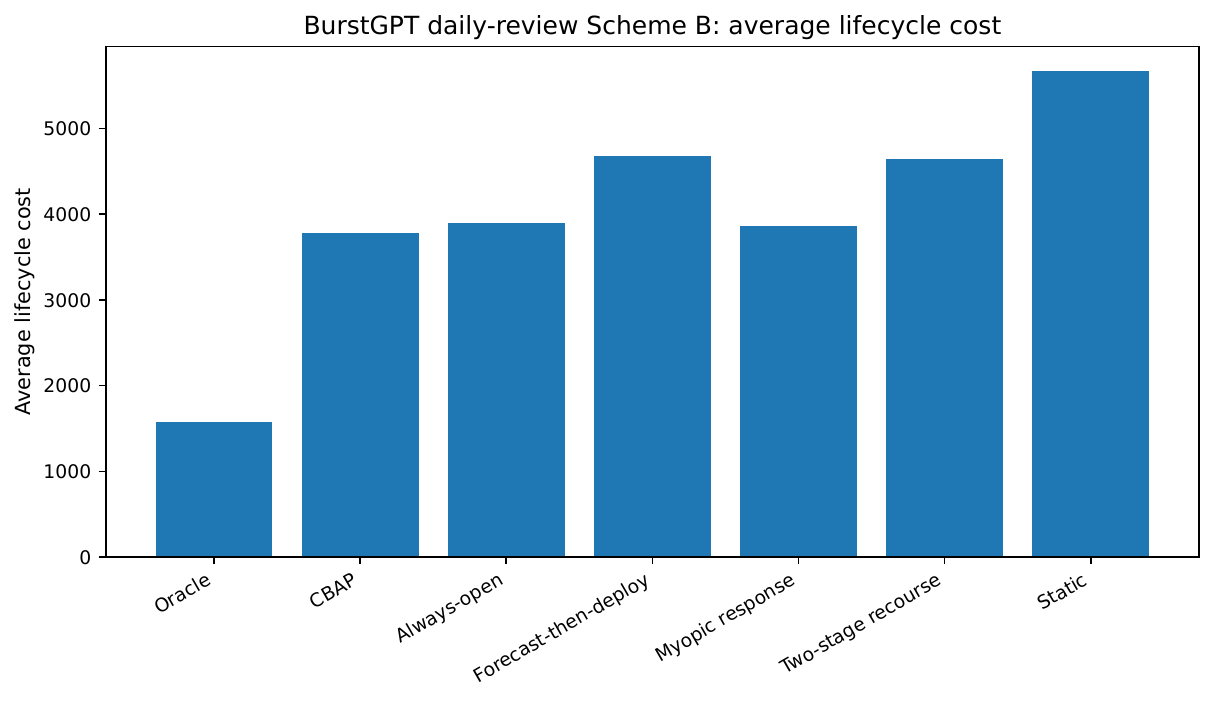}
\label{fig:burstgpt_daily_lifecycle_cost}
}
\subfigure[Cost decomposition]{
\includegraphics[width=3in]{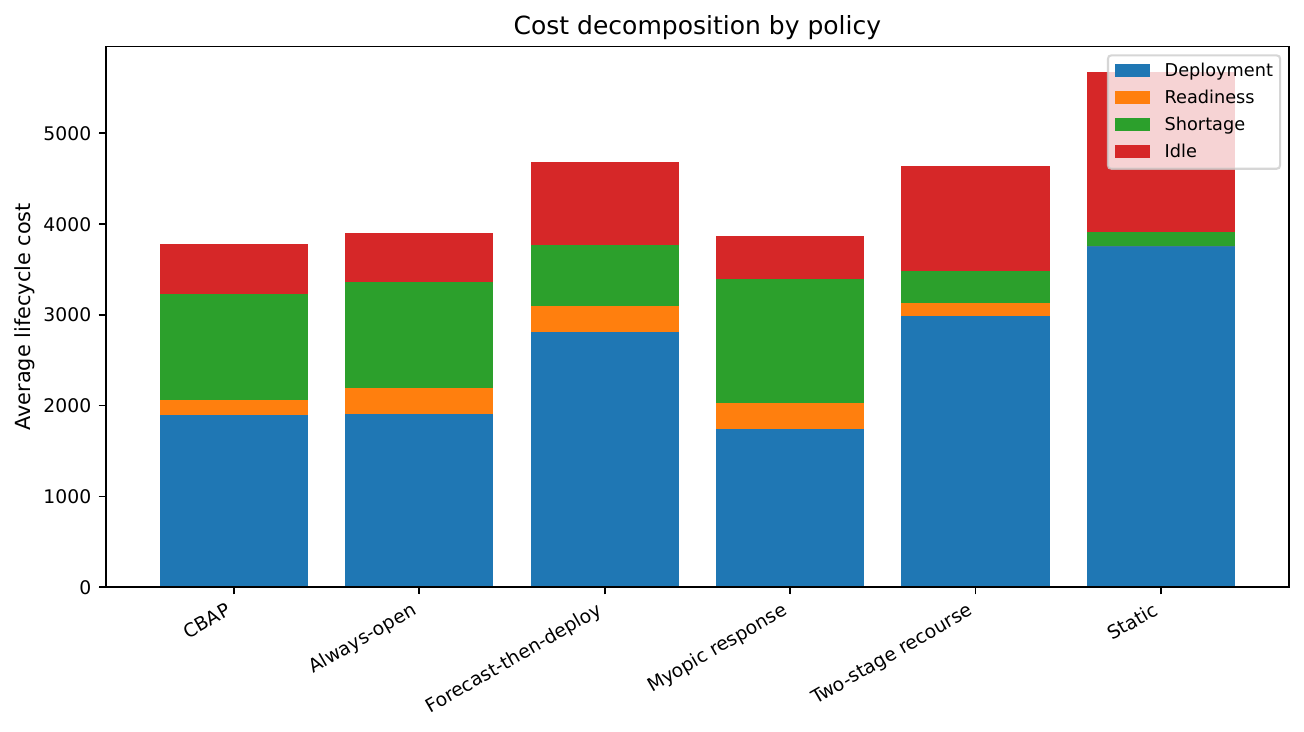}
\label{fig:burstgpt_daily_cost_decomposition}
}
\subfigure[Four-action mode shares]{
\includegraphics[width=3in]{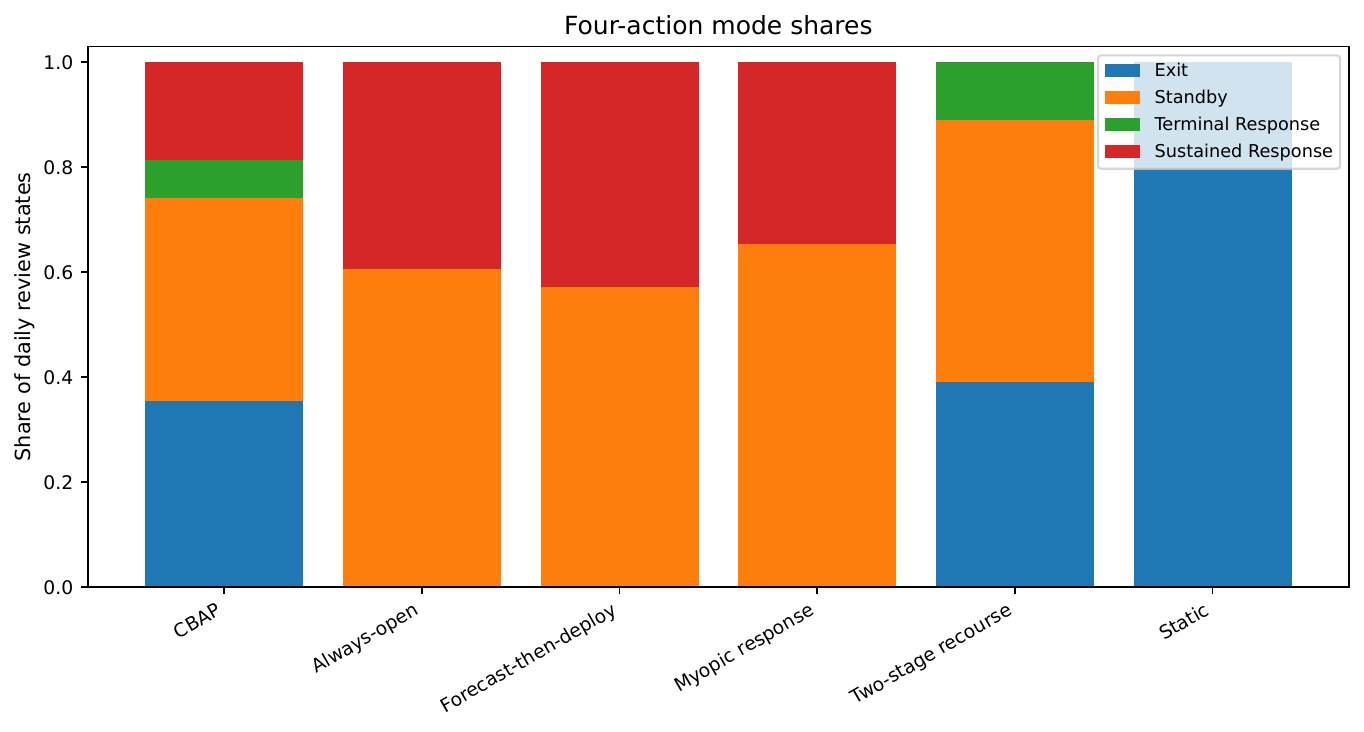}
\label{fig:burstgpt_daily_four_action_modes}
}
\caption{BurstGPT daily-review results under calibrated capacity costs.}
\label{fig:burstgpt_daily_main}
\end{figure}

The appendices report normalized regret, workload-intensity heterogeneity, readiness-cost sensitivity, and the timing diagnostic for Two-stage recourse. The results show that closure is most common in low-intensity windows, whereas preservation and positive response become more important as workload intensity rises.

\section{Conclusion}
\label{sec6}

This paper develops a response-readiness framework for dynamic service-capacity deployment in cold-start AI launches. The preservation--deployment decomposition separates current deployment from the decision to preserve future response feasibility and yields one preservation boundary, two side-specific deployment boundaries, and a conditional base-capacity rule. CBAP operationalizes this structure using scenarios drawn from available predictive workload beliefs, with decision and lifecycle performance bounds. The numerical experiment and BurstGPT-based application evaluate the decision layer and indicate that, in the tested settings, endogenous readiness control lowers lifecycle cost relative to full-window commitment policies by limiting premature deployment and unnecessary readiness spending. The main managerial implication is that deployment and readiness are distinct decisions: zero deployment may represent Standby rather than Exit, while positive deployment need not imply continued readiness.

The current model focuses on a single product and represents configured compute resources through an aggregate serviceable-capacity state. Future research could consider multi-region deployment, heterogeneous GPU types, capacity-reservation portfolios, explicit service-level constraints, and shared compute pools across products. The response-readiness logic may also extend to digital services, online campaigns, emergency capacity planning, and other short-lifecycle operations in which firms repeatedly update demand information while maintaining costly response options.
\bibliographystyle{informs2014}
\bibliography{ref1}
\ECSwitch
\ECHead{\centering Online Appendix to\\
``Dynamic Cloud Service-Capacity Deployment with Costly Response Readiness''}

\makeatletter

\setcounter{section}{0}
\setcounter{subsection}{0}
\setcounter{subsubsection}{0}

\renewcommand{\thesection}{\Alph{section}}
\renewcommand{\thesubsection}{\thesection.\arabic{subsection}}
\renewcommand{\thesubsubsection}{\thesubsection.\arabic{subsubsection}}

\renewcommand{\theHsection}{EC.\Alph{section}}
\renewcommand{\theHsubsection}{EC.\Alph{section}.\arabic{subsection}}
\renewcommand{\theHsubsubsection}{EC.\Alph{section}.\arabic{subsection}.\arabic{subsubsection}}

\setcounter{equation}{0}
\renewcommand{\theequation}{\thesection.\arabic{equation}}
\renewcommand{\theHequation}{EC.\Alph{section}.\arabic{equation}}
\@addtoreset{equation}{section}

\setcounter{figure}{0}
\setcounter{table}{0}
\setcounter{subfigure}{0}
\setcounter{algorithm}{0}

\renewcommand{\thefigure}{\thesection.\arabic{figure}}
\renewcommand{\thetable}{\thesection.\arabic{table}}
\renewcommand{\thealgorithm}{\thesection.\arabic{algorithm}}

\renewcommand{\theHfigure}{EC.\Alph{section}.\arabic{figure}}
\renewcommand{\theHtable}{EC.\Alph{section}.\arabic{table}}
\renewcommand{\theHalgorithm}{EC.\Alph{section}.\arabic{algorithm}}

\makeatother
\section{Technical Support for Section~\ref{sec3}}
\label{app:sec3_proofs}

\subsection{Proof of Proposition~\ref{prop:scalar_boundaries}}
\label{app:proof_scalar_boundaries}

\noindent\emph{Proof.}
We first establish monotonicity of the conditional deployment quantity. Fix $a\in\{0,1\}$ and let $m'\ge m$. Write $q=q_r^{a*}(m,I)$ and $q'=q_r^{a*}(m',I)$, selecting the smallest minimizer at each state. Suppose for contradiction that $q'<q$. Decreasing differences implies
\[
O_r^a(q;m',I)-O_r^a(q';m',I)
\le
O_r^a(q;m,I)-O_r^a(q';m,I)
\le 0,
\]
where the last inequality follows because $q$ minimizes $O_r^a(\cdot;m,I)$. Because $q'$ minimizes $O_r^a(\cdot;m',I)$, the first difference is nonnegative. Both inequalities must hold with equality, so $q'$ also minimizes $O_r^a(\cdot;m,I)$. This contradicts the selection of $q$ as the smallest minimizer because $q'<q$. Therefore, $q_r^{a*}(m,I)$ is nondecreasing in $m$.

By Assumption~\ref{ass:preservation_single_crossing}, $G_r(m,I)$ is continuous in $m$ and is strictly increasing on the relevant preservation crossing interval specified in Assumption~\ref{ass:interior_crossings}. Since
\[
G_r(\underline m_r^P,I)<k_r<G_r(\overline m_r^P,I),
\]
the intermediate-value theorem gives a solution to $G_r(m,I)=k_r$. Strict monotonicity on the crossing interval makes the solution unique; denote it by $m_r^P(k_r;I)$.

For $a\in\{0,1\}$, Assumption~\ref{ass:quantity_value_regular} states that $\psi_r^a(m;I)$ is continuous and strictly decreasing in $m$. Since
\[
\psi_r^a(\underline m_r^a;I)>0>\psi_r^a(\overline m_r^a;I),
\]
the intermediate-value theorem and strict monotonicity give a unique solution to $\psi_r^a(m;I)=0$, denoted by $m_r^{D,a}(I)$.

Consider the conditional deployment decision on side $a$. By Assumption~\ref{ass:quantity_value_regular}, $O_r^a(q;m,I)$ is convex in $q$ on $[0,\bar q_r]$. The smallest minimizer is zero if and only if the right derivative at zero is nonnegative:
\[
q_r^{a*}(m,I)=0
\quad\Longleftrightarrow\quad
\psi_r^a(m;I)\ge0.
\]
Because $\psi_r^a(m;I)$ is strictly decreasing, this is equivalent to
\[
q_r^{a*}(m,I)=0
\quad\Longleftrightarrow\quad
m\le m_r^{D,a}(I),
\]
with the boundary case assigned to zero deployment by the tie-breaking rule. Hence $q_r^{a*}(m,I)>0$ if and only if $m>m_r^{D,a}(I)$.

By \eqref{eq:preservation_deployment_rule}, the channel closes if and only if $G_r(m,I)\le k_r$, or equivalently $m\le m_r^P(k_r;I)$. On the close-after-current side, the selected quantity is $q_r^{0*}(m,I)$, giving Exit when $m\le m_r^{D,0}(I)$ and Terminal Response otherwise. The channel is preserved when $m>m_r^P(k_r;I)$. On this side, quantity $q_r^{1*}(m,I)$ gives Standby when $m\le m_r^{D,1}(I)$ and Sustained Response otherwise. This establishes \eqref{eq:four_region_policy}.

Fix $k_r'>k_r$ while holding $k_{r+1},\ldots,k_R$ and all other primitives fixed. The functions $A_r^0$, $A_r^1$, and $G_r=A_r^0-A_r^1$ exclude the current readiness payment. Therefore, changing $k_r$ leaves $G_r(\cdot,I)$ unchanged. Since
\[
G_r(m_r^P(k_r;I),I)=k_r
\quad\text{and}\quad
G_r(m_r^P(k_r';I),I)=k_r',
\]
strict monotonicity of $G_r$ on the crossing interval implies
\[
m_r^P(k_r';I)>m_r^P(k_r;I).
\]
The deployment thresholds $m_r^{D,0}(I)$ and $m_r^{D,1}(I)$ solve $\psi_r^0(m;I)=0$ and $\psi_r^1(m;I)=0$. Current readiness cost $k_r$ is added outside $O_r^1$ and does not enter $O_r^0$. Holding future readiness costs fixed, it therefore does not affect $\psi_r^0$, $\psi_r^1$, $m_r^{D,0}(I)$, or $m_r^{D,1}(I)$. \hfill$\square$

\subsection{Proof of Corollary~\ref{cor:base_capacity}}
\label{app:base_capacity_representation}

\noindent\emph{Proof.}
Fix review $r$, scalar belief $m$, capacity state $I$, and continuation side $a\in\{0,1\}$. Let $y=I+q$ be post-deployment capacity. Under the assumed dynamics, $g_r(I,q)=\widetilde g_r(I+q)=\widetilde g_r(y)$. Served workload and the next scalar belief are $X_r(y)=\min\{D_r,y\}$ and $m_{r+1}(y)=\mathcal B_r(m,X_r(y),y)$. Therefore,
\[
O_r^a(q;m,I)
=
c_rq
+
\mathbb E_m
\left[
h(I+q-D_r)^+
+
b(D_r-I-q)^+
+
V_{r+1}^a
\left(
m_{r+1}(I+q),
\widetilde g_r(I+q)
\right)
\right].
\]
Using $y=I+q$ and $c_rq=c_r(y-I)$, we obtain
\[
O_r^a(q;m,I)
=
-c_rI+\Phi_r^a(y;m),
\qquad y=I+q.
\]
The term $-c_rI$ is fixed at the current state and does not affect minimization over $q$.

Because $q\in[0,\bar q_r]$, post-deployment capacity satisfies $y\in[I,I+\bar q_r]$. Hence,
\[
\min_{q\in[0,\bar q_r]}O_r^a(q;m,I)
\]
is equivalent to
\[
\min_{y\in[I,I+\bar q_r]}\Phi_r^a(y;m).
\]
By Assumption~\ref{ass:base_capacity}, $\Phi_r^a(\cdot;m)$ is convex and attains a minimum. Because $\bar y_r^a(m)$ is its smallest global minimizer, the smallest minimizer over $[I,I+\bar q_r]$ is the projection of $\bar y_r^a(m)$ onto that interval:
\[
y_r^{a*}(m,I)
=
\min
\left\{
I+\bar q_r,
\max\{I,\bar y_r^a(m)\}
\right\}.
\]
Therefore,
\[
q_r^{a*}(m,I)
=
y_r^{a*}(m,I)-I
=
\min
\left\{
\bar q_r,
[\bar y_r^a(m)-I]^+
\right\}.
\]
If $\bar q_r=\infty$, then $q_r^{a*}(m,I)=[\bar y_r^a(m)-I]^+$.

To establish the dynamic critical-fractile representation, define $\Xi_r^a(y;m)$ as in \eqref{eq:Xi_base_capacity}. Then
\[
\Phi_r^a(y;m)
=
c_ry
+
\mathbb E_m
\left[
h(y-D_r)^+
+
b(D_r-y)^+
\right]
+
\Xi_r^a(y;m).
\]
Because $F_r(\cdot\mid m)$ is continuous,
\[
\frac{\partial}{\partial y}
\mathbb E_m[(y-D_r)^+]
=
F_r(y\mid m),
\]
and
\[
\frac{\partial}{\partial y}
\mathbb E_m[(D_r-y)^+]
=
-\left(1-F_r(y\mid m)\right).
\]
If $\bar y_r^a(m)$ is an interior minimizer and $\Xi_r^a(\cdot;m)$ is differentiable at $\bar y_r^a(m)$, the first-order condition for minimizing $\Phi_r^a(y;m)$ is
\[
0
=
c_r
+
hF_r(\bar y_r^a(m)\mid m)
-
b\left(1-F_r(\bar y_r^a(m)\mid m)\right)
+
\partial_y\Xi_r^a(\bar y_r^a(m);m).
\]
Rearranging yields
\[
F_r(\bar y_r^a(m)\mid m)
=
\frac{
b-c_r-\partial_y\Xi_r^a(\bar y_r^a(m);m)
}{
b+h
}.
\]
At a kink point, the same argument gives the inclusion in \eqref{eq:dynamic_critical_fractile_subgradient}. Boundary minimizers satisfy the corresponding one-sided or normal-cone condition. \hfill$\square$

\subsection{Supplementary Bound for Always-Open Policies}
\label{app:always_open_readiness_loss}

This subsection bounds the cost of an externally required full-window readiness commitment. Let $V_r^\pi(\mathcal S_r)$ denote expected remaining lifecycle cost from state $\mathcal S_r=(\mathbb P_r,I_r,o_r)$ under policy $\pi$, and let $\pi^*$ be an unrestricted optimal policy. Let $\Pi^{\mathrm{AO}}$ contain policies that keep the channel open through every remaining scheduled review, so $o_{s+1}=1$ for $s=r,\ldots,R$. Under this commitment, $k_R$ is the readiness payment assigned to the final evaluation interval, not the value of a separately modeled post-horizon action. The optimal Always-open value is
\begin{equation}
\label{eq:always_open_policy_definition}
V_r^{\mathrm{AO}}(\mathcal S_r)
=
\inf_{\pi\in\Pi^{\mathrm{AO}}}
V_r^\pi(\mathcal S_r).
\end{equation}
This constrained problem has its own Bellman recursion. With $V_{R+1}^{\mathrm{AO}}=0$, for $s=r,\ldots,R$,
\begin{equation}
\label{eq:always_open_recursion}
V_s^{\mathrm{AO}}(\mathbb P_s,I_s)
=
k_s
+
\min_{q\in\mathcal Q_s}
\left\{
c_sq
+
\mathbb E_{\mathbb P_s}
\left[
\ell_s(I_s,q,D_s)
+
V_{s+1}^{\mathrm{AO}}
\left(
\mathbb P_{s+1}(q),
I_{s+1}(q)
\right)
\mid
\mathbb P_s,I_s
\right]
\right\}.
\end{equation}

\begin{proposition}[Always-open readiness loss]
\label{prop:readiness_regret_app}
Suppose the response channel is open at state $\mathcal S_r$. Let the first closure time be
\begin{equation}
\label{eq:tau_closure_definition}
\tau_C^*
=
\inf\{s\in\{r,\ldots,R\}:o_{s+1}^{\pi^*}=0\},
\end{equation}
with $\tau_C^*=\infty$ if $\pi^*$ never closes the response channel. Then
\begin{equation}
\label{eq:always_open_regret_bound}
0
\le
V_r^{\mathrm{AO}}(\mathcal S_r)
-
V_r^{\pi^*}(\mathcal S_r)
\le
\mathbb E_{\pi^*}
\left[
\sum_{s=r}^{R} k_s
\mathbf{1}\{\tau_C^*\le s\}
\mid
\mathcal S_r
\right].
\end{equation}
Equivalently, since readiness costs are deterministic,
\begin{equation}
\label{eq:always_open_probability_bound}
V_r^{\mathrm{AO}}(\mathcal S_r)
-
V_r^{\pi^*}(\mathcal S_r)
\le
\sum_{s=r}^{R}
k_s
\Pr_{\pi^*}
\left(
\tau_C^*\le s
\mid
\mathcal S_r
\right).
\end{equation}
\end{proposition}

\noindent\emph{Proof.}
The lower bound follows because $\pi^*$ is optimal over the unrestricted policy class. For the upper bound, we construct a feasible Always-open policy.

For every $s<\tau_C^*$, define $\widetilde\pi$ to follow $\pi^*$. If $\tau_C^*<\infty$, then at epoch $\tau_C^*$, policy $\pi^*$ chooses quantity $q_{\tau_C^*}^{\pi^*}$ and sets $o_{\tau_C^*+1}^{\pi^*}=0$. Policy $\widetilde\pi$ chooses the same quantity but sets $o_{\tau_C^*+1}=1$. For $\tau_C^*<s\le R$, it keeps the channel open and chooses zero deployment. If $\tau_C^*=\infty$, it follows $\pi^*$. Hence, $\widetilde\pi\in\Pi^{\mathrm{AO}}$.

The two policies induce the same deployment quantities, service outcomes, capacity transitions, and belief updates on every sample path. At the closure epoch, they differ only in the next channel state. Thereafter, both choose zero deployment, but $\widetilde\pi$ keeps the channel open. Because service outcomes and observations depend on $(I_s,q_s,D_s)$ and capacity follows the same transition $g_s(I_s,q_s)$, their physical and belief paths remain identical. The only difference is that $\widetilde\pi$ pays readiness costs at epochs $s$ such that $\tau_C^*\le s$. Therefore,
\[
V_r^{\widetilde\pi}(\mathcal S_r)
-
V_r^{\pi^*}(\mathcal S_r)
=
\mathbb E_{\pi^*}
\left[
\sum_{s=r}^{R} k_s
\mathbf{1}\{\tau_C^*\le s\}
\mid
\mathcal S_r
\right].
\]
By the definition of the infimum in \eqref{eq:always_open_policy_definition},
\[
V_r^{\mathrm{AO}}(\mathcal S_r)
\le
V_r^{\widetilde\pi}(\mathcal S_r).
\]
Combining the last two displays gives the upper bound. Because $k_s$ is deterministic,
\[
\mathbb E_{\pi^*}
\left[
\sum_{s=r}^{R} k_s
\mathbf{1}\{\tau_C^*\le s\}
\mid
\mathcal S_r
\right]
=
\sum_{s=r}^{R}
k_s
\Pr_{\pi^*}
\left(
\tau_C^*\le s
\mid
\mathcal S_r
\right),
\]
which proves the probability form of the bound. \hfill$\square$

\section{Technical Support for Section~\ref{sec4}}
\label{app:sec4_proofs}
\subsection{Backward Training of Fitted Continuation Values}
\label{app:fitted_continuation_training}

When exact learned-model Bellman evaluation is computationally costly, CBAP can use fitted approximations of $V_r^0$ and $V_r^1$. Let
\[
\widetilde V_r^a
\left(
\widehat{\mathbb P}_r,
I_r;
\theta_r^a
\right),
\qquad a\in\{0,1\},
\]
denote value-function approximators trained backward over review epochs. Depending on the state representation and sample size, they may use neural networks, regression models, tree-based learners, or other supervised methods. A direct closed-channel rollout may replace $\widetilde V_r^0$, which is included here for notational symmetry.

For each epoch $r$, construct training states
\[
\left\{
(\widehat{\mathbb P}_{r,j},I_{r,j})
\right\}_{j=1}^{J_r}
\]
from historical or simulated lifecycles by rolling candidate policies forward and recording beliefs and capacities at review $r$. Additional capacity states may be sampled over the relevant operating range to improve coverage near decision boundaries. Repeated application of $g_r(I_r,q_r)$ generates these states under the non-depleting transition, while service outcomes depend on the sampled workload path.

Given trained continuation estimators $\widetilde V_{r+1}^0$ and $\widetilde V_{r+1}^1$, each training state receives two Bellman labels. At training state $j$, define the side-specific quantity estimate as
\begin{equation}
\label{eq:training_state_quantity_value_sec4}
\widehat O_{r,j}^a(q)
=
\widehat O_r^a
\left(
q;
\widehat{\mathbb P}_{r,j},
I_{r,j}
\right),
\qquad
a\in\{0,1\},
\quad
q\in\mathcal Q_r^G.
\end{equation}
Because no further deployment is feasible after closure, the closed-channel label fixes the current quantity at zero:
\begin{equation}
\label{eq:closed_bellman_label_sec4}
Y_{r,j}^0
=
\widehat O_{r,j}^0(0).
\end{equation}
The open-channel label is
\begin{equation}
\label{eq:open_bellman_label_sec4}
Y_{r,j}^1
=
\min
\left\{
\min_{q\in\mathcal Q_r^G}
\widehat O_{r,j}^0(q),
\;
k_r+
\min_{q\in\mathcal Q_r^G}
\widehat O_{r,j}^1(q)
\right\}.
\end{equation}
The estimators are fitted by minimizing
\begin{equation}
\label{eq:fitted_value_loss_by_side}
(\widehat\theta_r^0,\widehat\theta_r^1)
\in
\arg\min_{\theta_r^0,\theta_r^1}
\mathcal L_r(\theta_r^0,\theta_r^1),
\end{equation}
where
\begin{equation}
\label{eq:fitted_value_loss_definition}
\mathcal L_r(\theta_r^0,\theta_r^1)
=
\frac{1}{J_r}
\sum_{j=1}^{J_r}
\sum_{a\in\{0,1\}}
\left[
\widetilde V_r^a
\left(
\widehat{\mathbb P}_{r,j},
I_{r,j};
\theta_r^a
\right)
-
Y_{r,j}^a
\right]^2
+
\lambda_{\mathrm{reg}}
\sum_{a\in\{0,1\}}
\|\theta_r^a\|_2^2.
\end{equation}
The implementation may represent the belief by samples, posterior parameters, or sufficient predictive summaries. This computational choice does not change the state definition in the dynamic program.

When fitted continuation values are used, the paths used to construct training, validation, and out-of-sample evaluation states should be strictly separated to prevent data leakage across temporal horizons. Validation states, rather than evaluation lifecycles, should determine hyperparameters, model architecture, and stopping rules.

\begin{algorithm}[htp]
\small
\caption{Backward Training of Fitted Continuation Values}
\label{alg:value_training}
\begin{algorithmic}
\STATE \textbf{Input:} Training belief-capacity states $\{(\widehat{\mathbb P}_{r,j},I_{r,j})\}_{j=1}^{J_r}$ for $r=1,\ldots,R$; deployment grids $\mathcal Q_r^G$; costs; label scenario sizes; value-function classes $\widetilde V_r^a(\cdot;\theta_r^a)$, $a\in\{0,1\}$; validation states for model selection.
\STATE Set $\widetilde V_{R+1}^0=\widetilde V_{R+1}^1=0$.
\FOR{$r=R,R-1,\ldots,1$}
\FOR{each training state $(\widehat{\mathbb P}_{r,j},I_{r,j})$}
\STATE Draw label-generation scenarios from $\widehat{\mathbb P}_{r,j}$.
\FOR{each $a\in\{0,1\}$ and each $q\in\mathcal Q_r^G$}
\STATE Compute and store $\widehat O_{r,j}^a(q)
=
\widehat O_r^a
\left(
q;
\widehat{\mathbb P}_{r,j},
I_{r,j}
\right)$ using Algorithm~\ref{alg:value_estimation}.
\ENDFOR
\STATE Set $Y_{r,j}^0=\widehat O_{r,j}^0(0)$.
\STATE Set $Y_{r,j}^1$ using \eqref{eq:open_bellman_label_sec4}.
\ENDFOR
\STATE Solve \eqref{eq:fitted_value_loss_by_side} and obtain $(\widehat\theta_r^0,\widehat\theta_r^1)$.
\STATE Set $\widetilde V_r^a(\cdot)=\widetilde V_r^a(\cdot;\widehat\theta_r^a)$ for $a\in\{0,1\}$.
\ENDFOR
\STATE \textbf{Return:} fitted continuation values 
$\left\{
\widetilde V_r^a
\left(
\cdot;
\widehat\theta_r^a
\right)
:
a\in\{0,1\},
\;
r=1,\ldots,R
\right\}$. 
\end{algorithmic}
\end{algorithm}
\subsection{Proof of Theorem~\ref{thm:cbap_regret}}
\label{app:proof_cbap_regret}

\noindent\emph{Proof.}
Fix epoch $r$ and reachable open-channel history $\mathcal H_r$. On $\mathcal E_r^{\mathrm{bel}}\cap\mathcal E_r^{\mathrm{saa}}\cap\mathcal E_r^{\mathrm{cont}}$, the triangle inequality gives, for every $u\in\mathcal K_r$,
\begin{equation}
\label{eq:proof_uniform_candidate_error}
\left|
\widehat Q_r(u)-Q_r^{\mathbb P}(u)
\right|
\le
\varepsilon_r^Q,
\qquad
\varepsilon_r^Q
=
e_r^{\mathrm{bel}}
+
e_r^{\mathrm{saa}}
+
e_r^{\mathrm{cont}}.
\end{equation}

Let $u_r^*$ minimize $Q_r^{\mathbb P}$ over $\mathcal K_r$, and let $\widehat u_r$ be the candidate selected by CBAP. If CBAP closes, $\widehat G_r\le k_r-\rho_r$ implies $\widehat A_r^0\le k_r+\widehat A_r^1-\rho_r$, so closure is the estimated minimizer. If CBAP preserves, $\widehat G_r>k_r-\rho_r$ implies $k_r+\widehat A_r^1<\widehat A_r^0+\rho_r$. Thus, before numerical error, the selected side is within $\rho_r$ of the estimated minimum. Allowing numerical optimization tolerance $\eta_{\mathrm{opt},r}$ gives
\[
\widehat Q_r(\widehat u_r)
\le
\min_{u\in\mathcal K_r}\widehat Q_r(u)
+
\eta_{\mathrm{opt},r}
+
\rho_r.
\]
Using \eqref{eq:proof_uniform_candidate_error},
\begin{align*}
Q_r^{\mathbb P}(\widehat u_r)
&\le \widehat Q_r(\widehat u_r)+\varepsilon_r^Q \\
&\le \widehat Q_r(u_r^*)+\eta_{\mathrm{opt},r}+\rho_r+\varepsilon_r^Q \\
&\le Q_r^{\mathbb P}(u_r^*)+2\varepsilon_r^Q+\eta_{\mathrm{opt},r}+\rho_r.
\end{align*}
The final difference is bounded by $E_r$, proving \eqref{eq:cbap_local_bound}.

For the expected lifecycle bound, define the Bellman gap of the true grid-restricted dynamic program as
\[
\operatorname{Gap}_r(\mathcal S_r,u)
=
Q_r^{\mathbb P}(\mathcal S_r,u)
-
\min_{v\in\mathcal K_r}
Q_r^{\mathbb P}(\mathcal S_r,v).
\]
The finite-horizon performance-difference identity for the randomized online policy is
\begin{equation}
\label{eq:expected_performance_difference}
V_{\mathbb P}^{\mathrm{CBAP}}(\mathcal S_1)
-
V_{\mathbb P}^{\pi_G^*}(\mathcal S_1)
=
\mathbb E_{\mathrm{CBAP}}
\left[
\sum_{r=1}^{R}
\operatorname{Gap}_r(\mathcal S_r,\widehat u_r)
\right].
\end{equation}
The expectation in \eqref{eq:expected_performance_difference} is over demand and CBAP's internal scenario draws.

Let $\mathcal E_r=\mathcal E_r^{\mathrm{bel}}\cap\mathcal E_r^{\mathrm{saa}}\cap\mathcal E_r^{\mathrm{cont}}$. By Assumption~\ref{ass:belief_stability} and the union bound,
\[
\Pr(\mathcal E_r^c)
\le
\beta_r+\gamma_r+\varphi_r.
\]
The bound $\Pr((\mathcal E_r^{\mathrm{saa}})^c)\le\gamma_r$ follows by taking expectations of $\Pr((\mathcal E_r^{\mathrm{saa}})^c\mid\mathcal H_r)\le\gamma_r$. On $\mathcal E_r$, the local result gives $\operatorname{Gap}_r(\mathcal S_r,\widehat u_r)\le E_r$. On $\mathcal E_r^c$, Assumption~\ref{ass:belief_stability} gives $\operatorname{Gap}_r(\mathcal S_r,\widehat u_r)\le\overline\Delta_r$. Therefore,
\[
\mathbb E_{\mathrm{CBAP}}
\left[
\operatorname{Gap}_r(\mathcal S_r,\widehat u_r)
\right]
\le
E_r
+
(\beta_r+\gamma_r+\varphi_r)\overline\Delta_r.
\]
At a closed-channel state, the feasible action set is a singleton and the Bellman gap is zero. Summing over $r$ in \eqref{eq:expected_performance_difference} proves \eqref{eq:cbap_regret_bound}. \hfill$\square$

\subsection{Mode-Selection Stability with Closure Buffering}
\label{app:mode_consistency_note}

The error decomposition also gives a local robustness result for mode selection. At a reachable open-channel history with $r<R$, define the true mode costs as
\begin{equation}
\label{eq:mode_costs_four}
\begin{aligned}
\mathrm{Cost}_r^{\mathrm{Exit}} 
&= Q_r^{\mathbb P}((0,0)), 
& \mathrm{Cost}_r^{\mathrm{Standby}} 
&= Q_r^{\mathbb P}((0,1)), \\
\mathrm{Cost}_r^{\mathrm{Term}} 
&= \inf_{q\in\mathcal Q_r^G,\,q>0} Q_r^{\mathbb P}((q,0)),
& \mathrm{Cost}_r^{\mathrm{Sust}} 
&= \inf_{q\in\mathcal Q_r^G,\,q>0} Q_r^{\mathbb P}((q,1)).
\end{aligned}
\end{equation}
The value $Q_r^{\mathbb P}((q,1))$ already includes readiness payment. Define estimated mode costs analogously using $\widehat Q_r$. If the grid contains no positive quantity, omit the two positive-deployment modes. Let $\Delta_r^{\mathrm{mode}}$ be the difference between the lowest and second-lowest true mode costs.

\begin{corollary}[Mode-selection consistency with closure buffering]
\label{cor:mode_consistency}
Suppose the estimated cost of each available mode differs from its true cost by at most $\varepsilon_r^Q$. If
\[
\Delta_r^{\mathrm{mode}}
>
2\varepsilon_r^Q+\rho_r+\eta_{\mathrm{opt},r},
\]
CBAP selects the optimal mode. For any $\Delta_r^{\mathrm{mode}}$, the one-step mode-cost loss is at most $2\varepsilon_r^Q+\rho_r+\eta_{\mathrm{opt},r}$. Under exact finite-grid enumeration, $\eta_{\mathrm{opt},r}=0$.
\end{corollary}

\noindent\emph{Proof.}
Let $u^*$ be the optimal available mode and $\widehat u$ the mode selected by CBAP. The closure buffer and optimization tolerance make $\widehat u$ an estimated $(\rho_r+\eta_{\mathrm{opt},r})$-minimizer:
\[
\widehat{\mathrm{Cost}}_r(\widehat u)
\le
\widehat{\mathrm{Cost}}_r(u^*)+\rho_r+\eta_{\mathrm{opt},r}.
\]
For any $u\ne u^*$, the uniform mode-cost error gives
\[
\widehat{\mathrm{Cost}}_r(u)-\widehat{\mathrm{Cost}}_r(u^*)
\ge
\mathrm{Cost}_r(u)-\mathrm{Cost}_r(u^*)-2\varepsilon_r^Q
\ge
\Delta_r^{\mathrm{mode}}-2\varepsilon_r^Q.
\]
If $\Delta_r^{\mathrm{mode}}>2\varepsilon_r^Q+\rho_r+\eta_{\mathrm{opt},r}$, every nonoptimal mode is more than $\rho_r+\eta_{\mathrm{opt},r}$ above the estimated best mode, so CBAP selects $u^*$. In general,
\begin{align*}
\mathrm{Cost}_r(\widehat u)
&\le \widehat{\mathrm{Cost}}_r(\widehat u)+\varepsilon_r^Q \\
&\le \widehat{\mathrm{Cost}}_r(u^*)+\rho_r+\eta_{\mathrm{opt},r}+\varepsilon_r^Q \\
&\le \mathrm{Cost}_r(u^*)+2\varepsilon_r^Q+\rho_r+\eta_{\mathrm{opt},r},
\end{align*}
which proves the loss bound. \hfill$\square$

\section{Additional Numerical Analyses}
\label{app:numerical_analyses}

\subsection{Additional Performance Diagnostics}
\label{app:additional_diagnostics}

Figure~\ref{fig:appendix_cost_regret} reports the lifecycle-cost and normalized-regret comparisons from Table~\ref{tab:baseline_results}. CBAP has normalized regret 0.047, compared with 0.091 for Always-open, 0.101 for Myopic response, and 0.145 for Forecast-then-deploy. The comparisons associate CBAP's advantage with both continuation-aware deployment and selective readiness preservation.

\begin{figure}[htbp]
\centering
\subfigure[Average lifecycle cost]{
\includegraphics[width=3in]{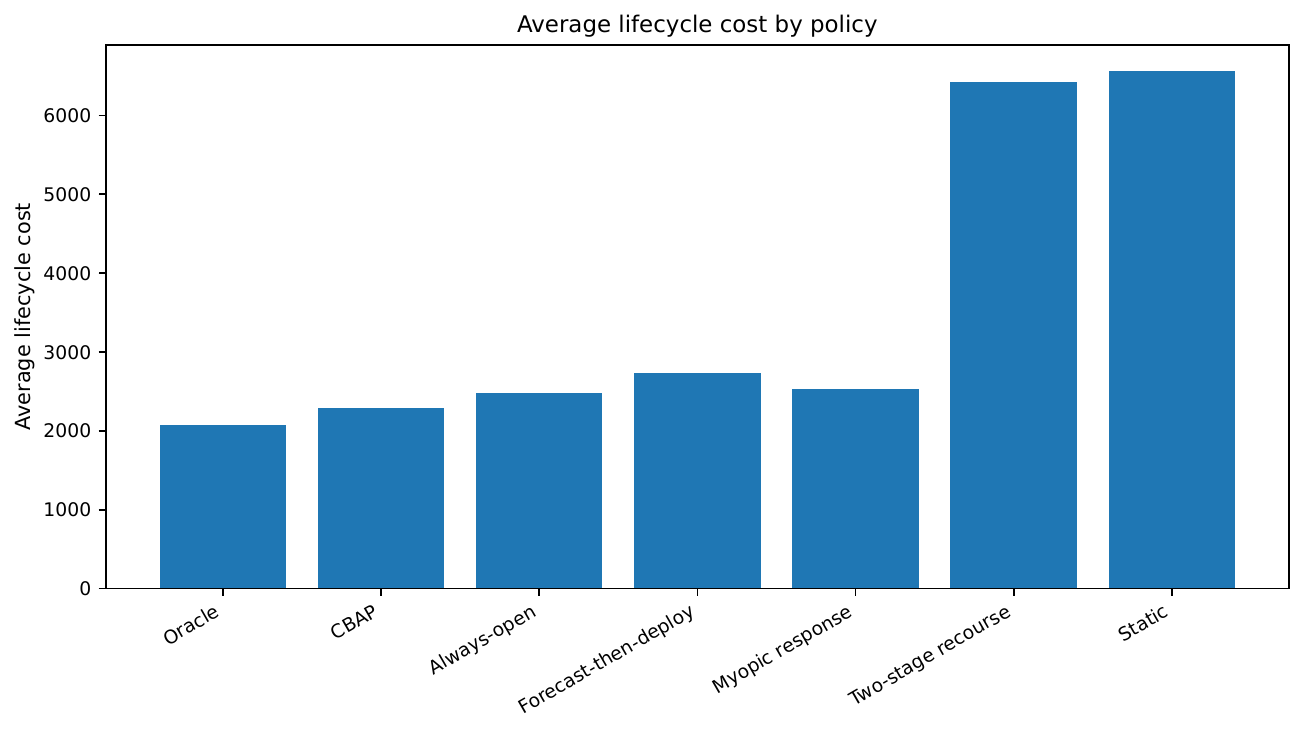}
\label{fig:appendix_avg_cost}
}
\subfigure[Normalized regret]{
\includegraphics[width=3in]{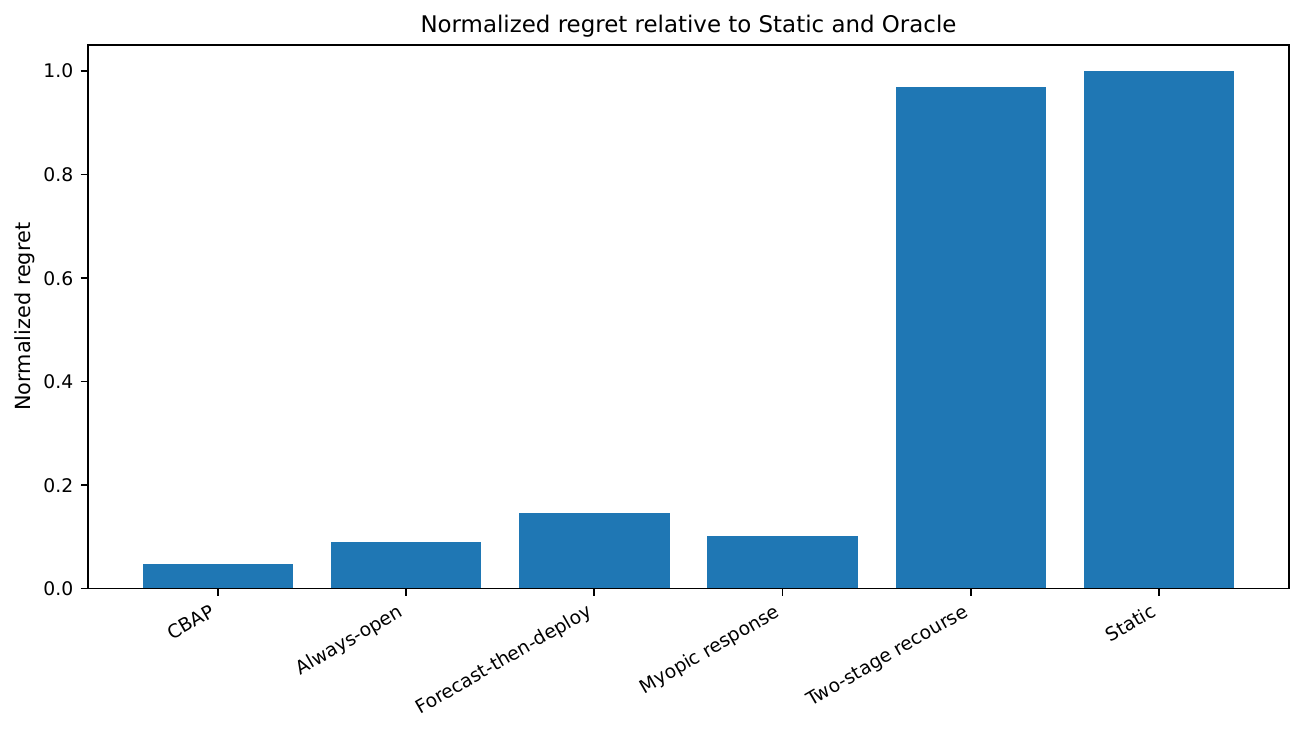}
\label{fig:appendix_regret}
}
\caption{Additional numerical diagnostics: lifecycle cost and normalized regret.}
\label{fig:appendix_cost_regret}
\end{figure}

Figure~\ref{fig:readiness_saving_appendix} shows how CBAP's readiness-cost saving relative to Always-open changes with $k_r$. The saving grows as readiness becomes more expensive because Always-open pays at every review, whereas CBAP closes when the estimated continuation benefit no longer covers the payment.

\begin{figure}[htbp]
\centering
\includegraphics[width=3in]{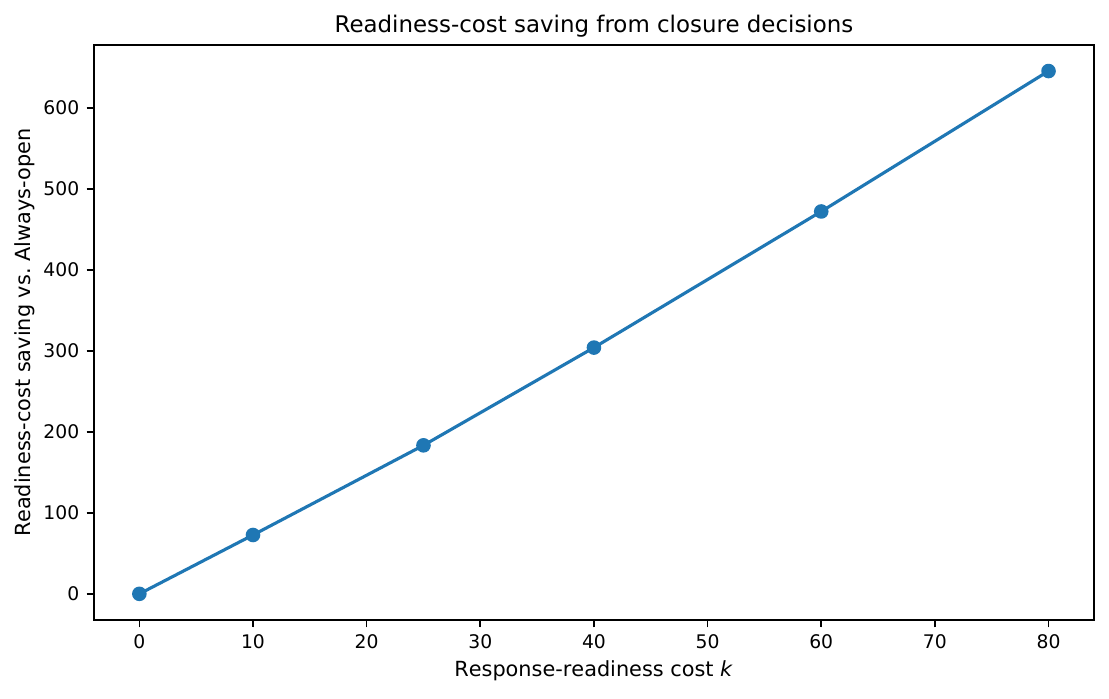}
\caption{Readiness-cost saving of CBAP relative to Always-open in the controlled numerical experiment.}
\label{fig:readiness_saving_appendix}
\end{figure}

Table~\ref{tab:mode_shares_appendix} reports the action shares underlying Figure~\ref{fig:mode_shares_four_action}. CBAP closes in 69.2\% of review states, preserves the channel in 30.9\%, and deploys positive capacity in 36.4\%. Standby is rare in this calibration, so preservation usually accompanies deployment. The benchmark shares reflect their policy restrictions: Always-open, Forecast-then-deploy, and Myopic response never close, while Static never deploys after launch.

\begin{table}[htbp]
\centering
\small
\caption{Four-action shares across review-time decisions in the numerical experiment}
\label{tab:mode_shares_appendix}
\renewcommand{\arraystretch}{1.15}
\begin{tabular}{lrrrr}
\toprule
\textbf{Policy} & \textbf{Exit} & \textbf{Standby} & \textbf{Terminal Resp.} & \textbf{Sustained Resp.} \\
\midrule
CBAP & 0.626 & 0.011 & 0.066 & 0.298 \\
Always-open & 0.000 & 0.517 & 0.000 & 0.483 \\
Forecast-then-deploy & 0.000 & 0.555 & 0.000 & 0.445 \\
Myopic response & 0.000 & 0.468 & 0.000 & 0.533 \\
Two-stage recourse & 0.500 & 0.400 & 0.100 & 0.000 \\
Static & 1.000 & 0.000 & 0.000 & 0.000 \\
\bottomrule
\end{tabular}
\end{table}

Figure~\ref{fig:numerical_mode_sensitivity_k} shows that higher readiness cost shifts probability from Standby and Sustained Response toward Exit and Terminal Response. The deployment and preservation margins therefore remain distinct: a higher $k_r$ changes whether the channel is kept open, but does not require current deployment to fall to zero.

\begin{figure}[htbp]
\centering
\includegraphics[width=3in]{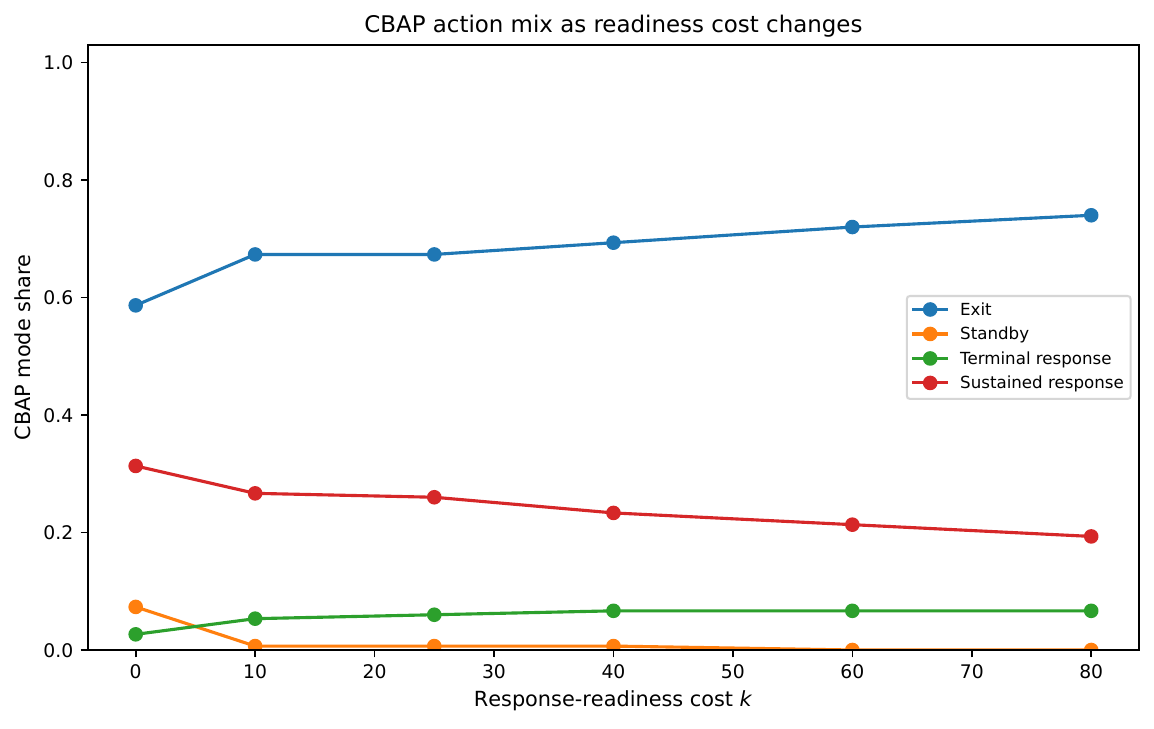}
\caption{CBAP four-action mode shares as readiness cost varies in the controlled numerical experiment.}
\label{fig:numerical_mode_sensitivity_k}
\end{figure}

\subsection{Performance by Workload Regime}
\label{app:regime_analysis}

Each simulated launch belongs to an unobserved Low, Regular, or Burst regime. Figure~\ref{fig:regime_cost_reduction} shows that the source of the cost reduction varies across regimes. On low-demand paths, CBAP limits unnecessary response and readiness spending. On regular and burst paths, it retains the ability to deploy after informative observations arrive.

\begin{figure}[htbp]
\centering
\subfigure[Average lifecycle cost by regime]{
\includegraphics[width=3in]{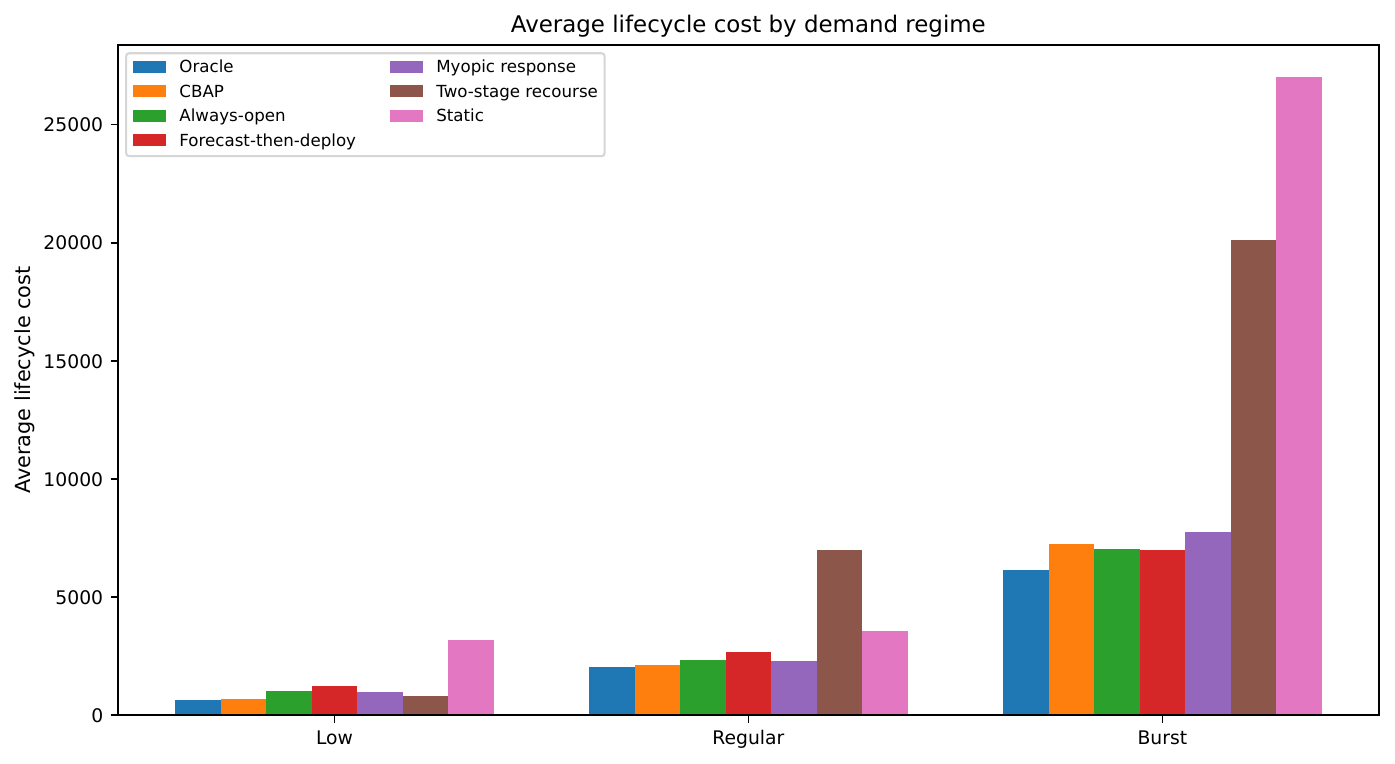}
\label{fig:regime_cost}
}
\subfigure[CBAP cost reduction by regime]{
\includegraphics[width=3in]{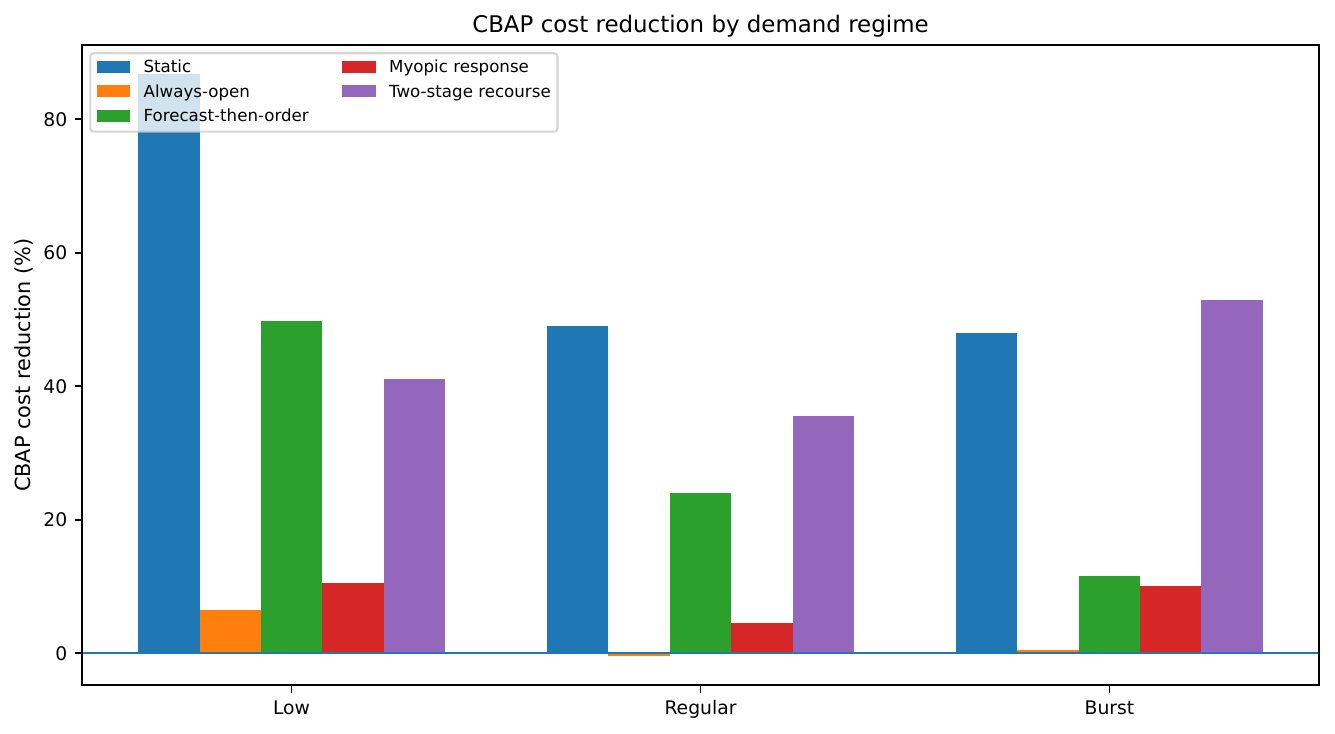}
\label{fig:regime_reduction}
}
\caption{Regime-level cost comparison in the controlled numerical experiment.}
\label{fig:regime_cost_reduction}
\end{figure}

Figure~\ref{fig:regime_modes_costs} and Table~\ref{tab:regime_mode_shares} report this adaptation. Exit accounts for 85.3\% of Low-regime decisions but 30.0\% in the Burst regime. Sustained Response rises from 10.4\% to 61.3\%, with the Regular regime between these cases. CBAP changes both deployment and preservation as evidence about the latent regime accumulates.

\begin{figure}[htbp]
\centering
\subfigure[CBAP four-action shares by regime]{
\includegraphics[width=3in]{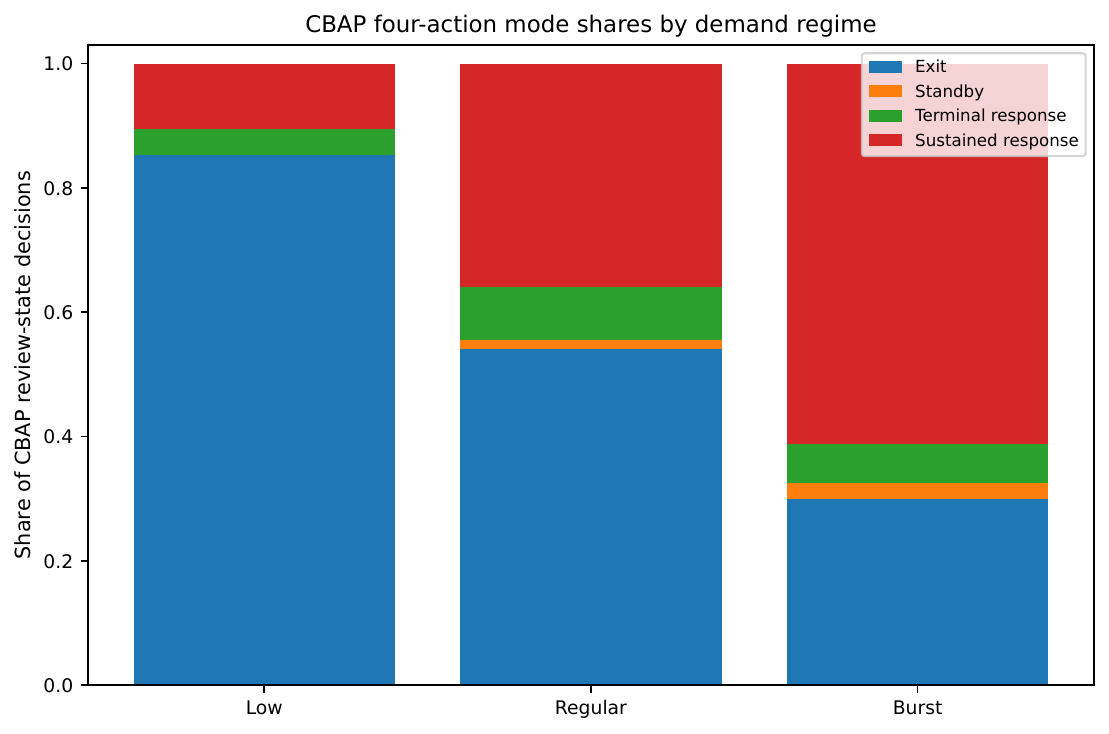}
\label{fig:regime_modes}
}
\subfigure[CBAP cost decomposition by regime]{
\includegraphics[width=3in]{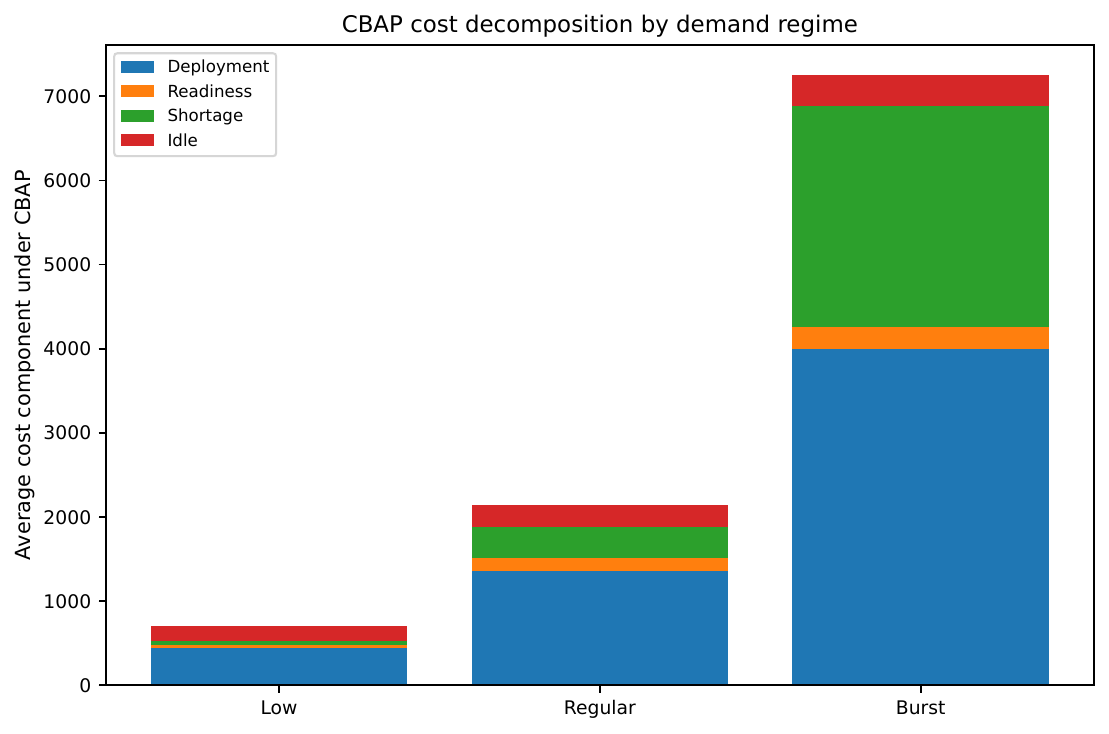}
\label{fig:regime_cbap_cost_decomposition}
}
\caption{Regime-level CBAP behavior in the controlled numerical experiment.}
\label{fig:regime_modes_costs}
\end{figure}

\begin{table}[htbp]
\centering
\small
\caption{CBAP four-action shares by latent demand regime}
\label{tab:regime_mode_shares}
\renewcommand{\arraystretch}{1.15}
\begin{tabular}{lrrrr}
\toprule
\textbf{Regime} & \textbf{Exit} & \textbf{Standby} & \textbf{Terminal Resp.} & \textbf{Sustained Resp.} \\
\midrule
Low & 0.853 & 0.000 & 0.042 & 0.104 \\
Regular & 0.541 & 0.015 & 0.085 & 0.359 \\
Burst & 0.300 & 0.025 & 0.063 & 0.613 \\
\bottomrule
\end{tabular}
\end{table}

\section{Additional BurstGPT Empirical Analyses}
\label{app:burstgpt_empirical_diagnostics}
\subsection{Normalized Regret and Cost Reduction}
\label{app:burstgpt_regret_cost_reduction}

Figure~\ref{fig:burstgpt_appendix_regret_reduction} complements Table~\ref{tab:burstgpt_main_results}. CBAP has normalized regret 0.536, compared with 0.559 for Myopic response and 0.567 for Always-open. Forecast-then-deploy and Two-stage recourse leave larger shares of the Static-to-Oracle gap, at 0.758 and 0.749. The two closest benchmarks permit repeated response; CBAP differs in how it values continuation and closure.

\begin{figure}[htbp]
\centering
\subfigure[Normalized regret]{
\includegraphics[width=3in]{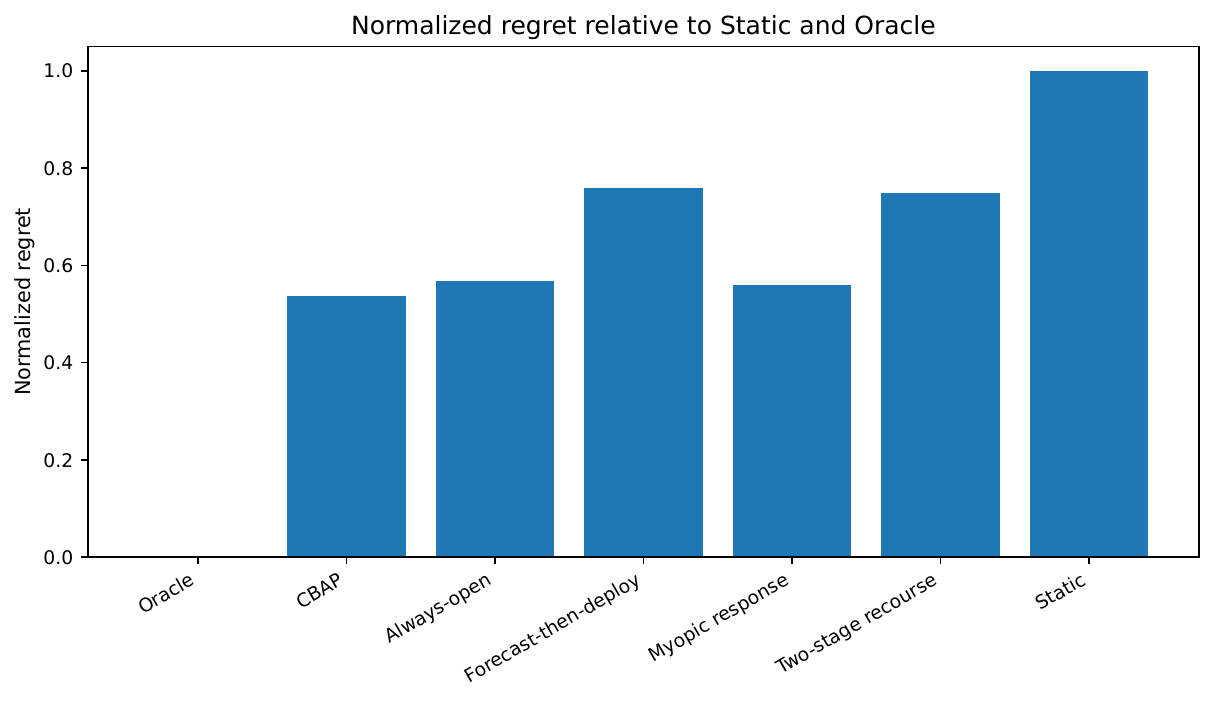}
\label{fig:burstgpt_appendix_regret}
}
\subfigure[Cost reduction relative to Static]{
\includegraphics[width=3in]{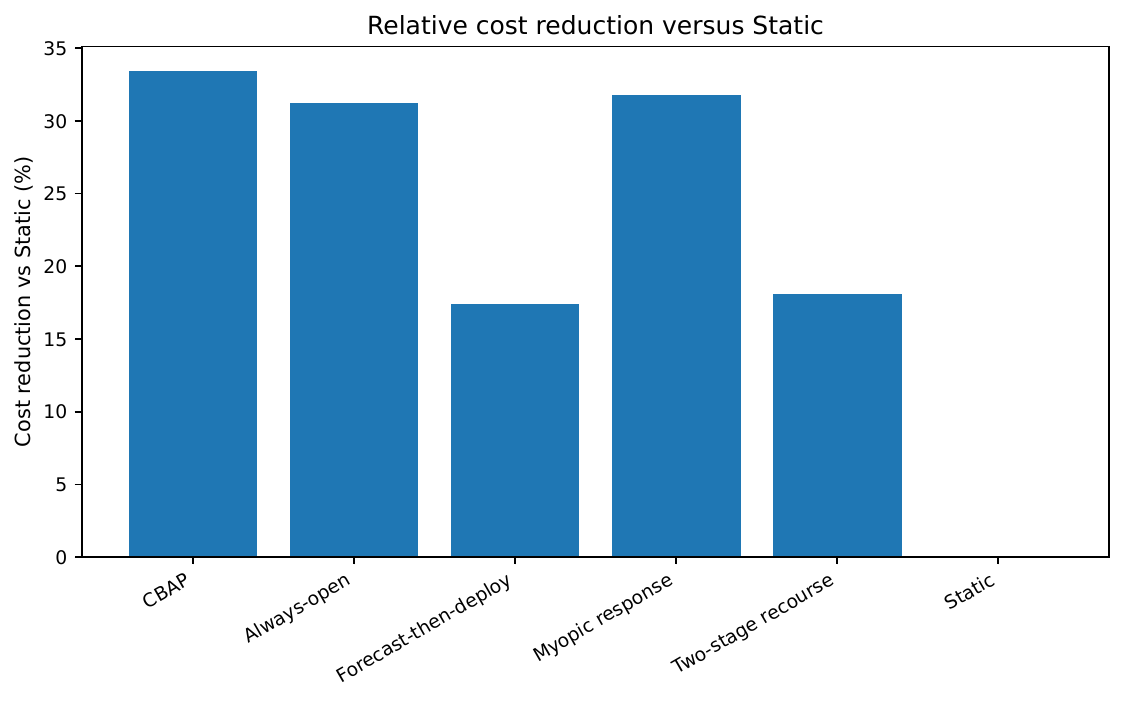}
\label{fig:burstgpt_cost_reduction}
}
\caption{Additional BurstGPT daily-review performance diagnostics.}
\label{fig:burstgpt_appendix_regret_reduction}
\end{figure}

Table~\ref{tab:burstgpt_daily_mode_shares} shows that CBAP uses all four modes. It makes no positive deployment in 74.2\% of review states, but more than half of those states are Standby rather than Exit. It also separates positive response into Terminal Response (7.2\%) and Sustained Response (18.6\%). In contrast, the always-open policies can use only Standby and Sustained Response, and Two-stage recourse cannot sustain the channel after its single deployment.

\begin{table}[htbp]
\centering
\small
\caption{Four-action shares in the BurstGPT daily-review application}
\label{tab:burstgpt_daily_mode_shares}
\renewcommand{\arraystretch}{1.15}
\begin{tabular}{lrrrr}
\toprule
\textbf{Policy} & \textbf{Exit} & \textbf{Standby} & \textbf{Terminal Resp.} & \textbf{Sustained Resp.} \\
\midrule
CBAP & 0.354 & 0.388 & 0.072 & 0.186 \\
Always-open & 0.000 & 0.607 & 0.000 & 0.393 \\
Forecast-then-deploy & 0.000 & 0.571 & 0.000 & 0.429 \\
Myopic response & 0.000 & 0.655 & 0.000 & 0.345 \\
Two-stage recourse & 0.391 & 0.500 & 0.109 & 0.000 \\
Static & 1.000 & 0.000 & 0.000 & 0.000 \\
\bottomrule
\end{tabular}
\end{table}

\subsection{Workload-Intensity Heterogeneity}
\label{app:burstgpt_heterogeneity}

We classify workload windows as Low, Medium, or High using total lifecycle demand. Figure~\ref{fig:burstgpt_intensity} and Table~\ref{tab:burstgpt_intensity_cbap} show how CBAP adjusts across these groups. In Low-intensity windows, Exit accounts for 62.3\% of decisions and average readiness cost is 68.45. In Medium and High windows, preservation accounts for 77.0\% and 72.0\% of decisions, respectively, and Standby is the largest mode. The High-intensity fill rate is 0.8060 despite frequent preservation, indicating that readiness does not eliminate service shortfalls on the most demanding paths.

\begin{figure}[htbp]
\centering
\subfigure[CBAP four-action shares by workload intensity]{
\includegraphics[width=3in]{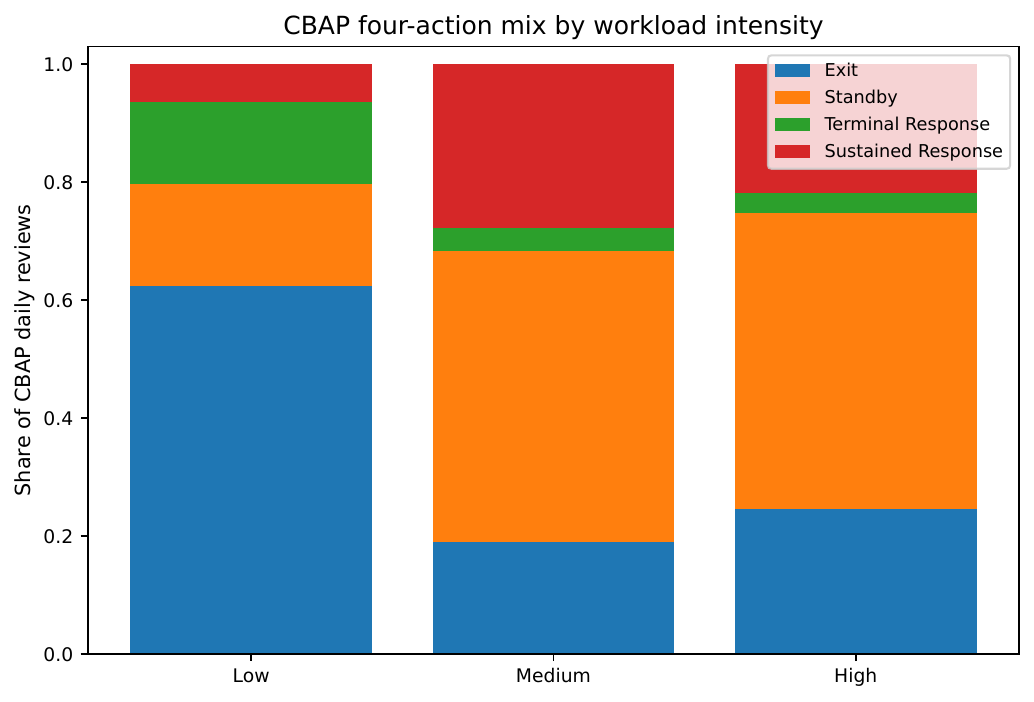}
\label{fig:burstgpt_intensity_modes}
}
\subfigure[Lifecycle cost by workload intensity]{
\includegraphics[width=3in]{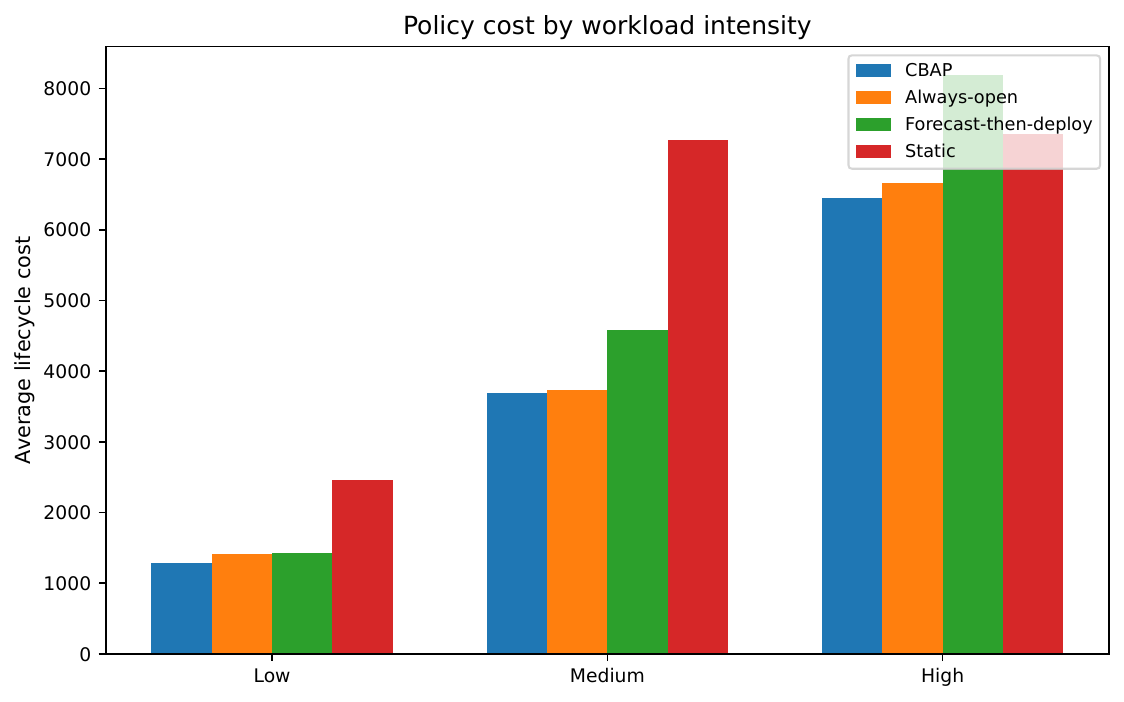}
\label{fig:burstgpt_intensity_cost}
}
\caption{Heterogeneity of CBAP behavior and lifecycle cost by workload intensity.}
\label{fig:burstgpt_intensity}
\end{figure}

\begin{table}[htbp]
\centering
\small
\caption{CBAP behavior by workload intensity in the BurstGPT daily-review application}
\label{tab:burstgpt_intensity_cbap}
\renewcommand{\arraystretch}{1.15}
\begin{tabular}{lrrrrrrr}
\toprule
\textbf{Intensity} & \textbf{Avg. cost} & \textbf{Fill rate} & \textbf{Ready. cost} & \textbf{Exit} & \textbf{Standby} & \textbf{Term.} & \textbf{Sust.} \\
\midrule
Low & 1287.90 & 0.9735 & 68.45 & 0.623 & 0.174 & 0.139 & 0.064 \\
Medium & 3697.39 & 0.9480 & 221.82 & 0.189 & 0.493 & 0.040 & 0.277 \\
High & 6450.28 & 0.8060 & 207.41 & 0.246 & 0.502 & 0.034 & 0.218 \\
\bottomrule
\end{tabular}
\end{table}

\subsection{Readiness-Cost Sensitivity}
\label{app:burstgpt_readiness_sensitivity}

We vary the readiness-intensity parameter $\alpha$ in
\begin{equation}
\label{eq:appendix_readiness_intensity}
k_r=16\alpha\bar c,
\end{equation}
where $\bar c=3$. The values $\alpha\in\{0,0.25,0.5,1.0,1.5,2.0\}$ correspond to $k_r\in\{0,12,24,48,72,96\}$.

Figure~\ref{fig:burstgpt_readiness_sensitivity} shows that CBAP closes more often as readiness becomes more expensive. Always-open is close to CBAP when $\alpha$ is low, but its relative cost rises with $\alpha$ because it cannot avoid scheduled readiness payments. The shift toward closure limits CBAP's readiness exposure as $\alpha$ increases.

\begin{figure}[htbp]
\centering
\subfigure[Lifecycle cost by readiness intensity]{
\includegraphics[width=3in]{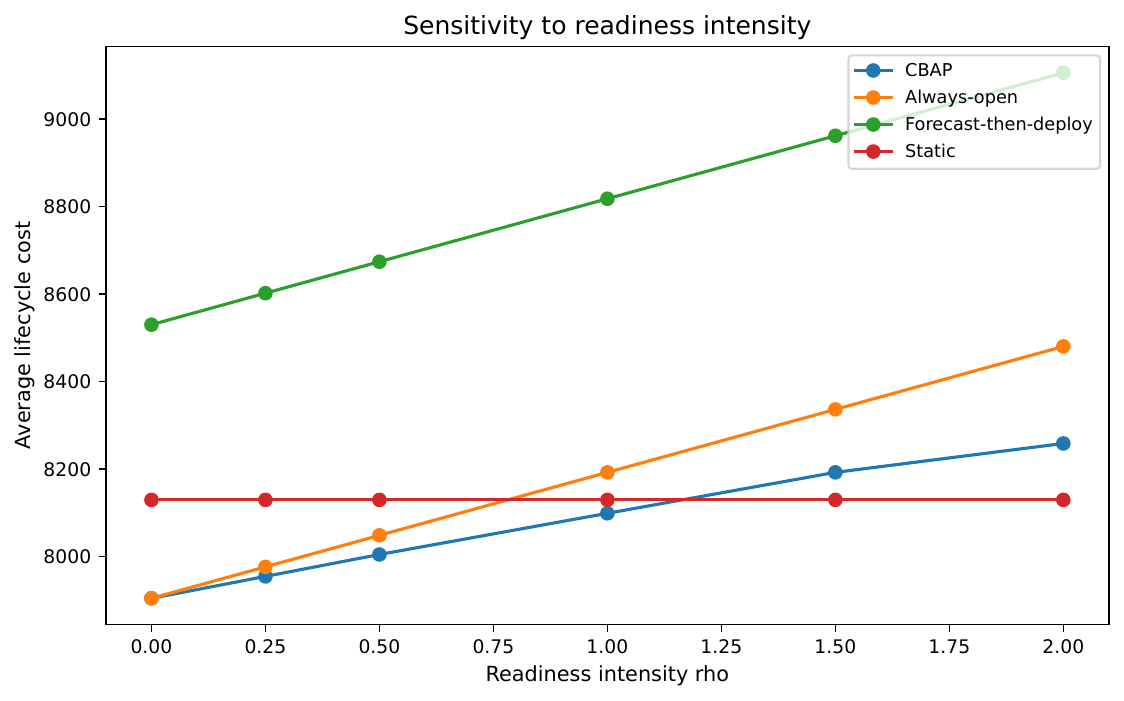}
\label{fig:burstgpt_readiness_cost}
}
\subfigure[CBAP four-action shares by readiness intensity]{
\includegraphics[width=3in]{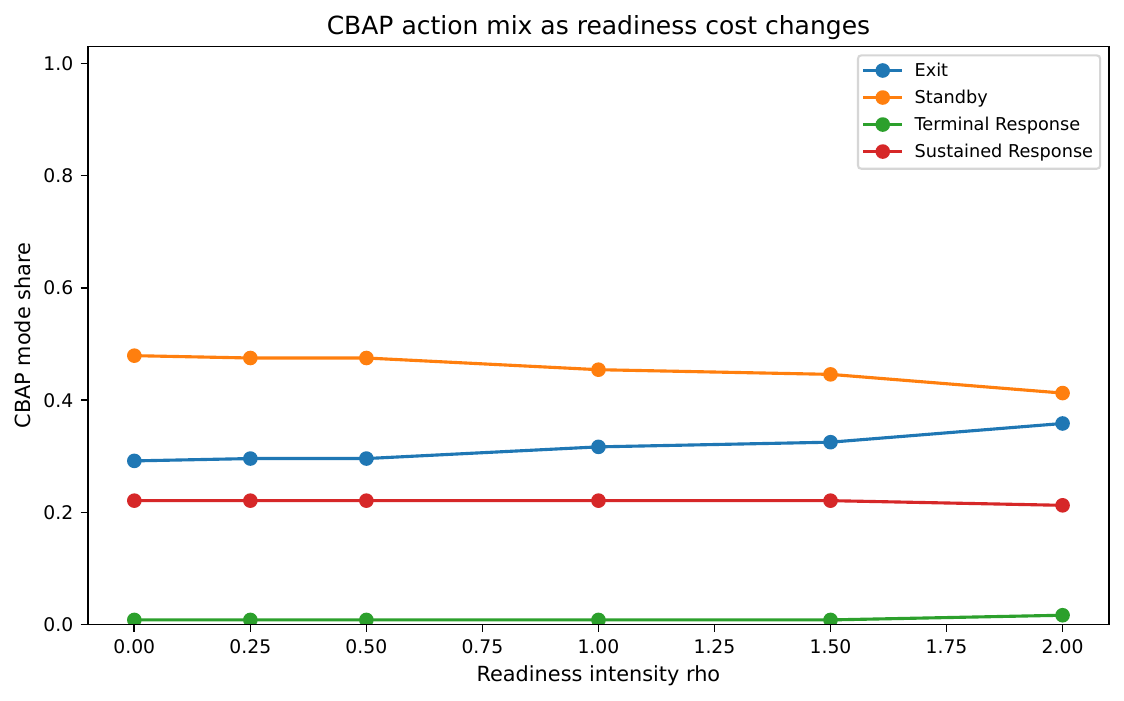}
\label{fig:burstgpt_readiness_modes}
}
\caption{Sensitivity to readiness intensity in the BurstGPT daily-review application.}
\label{fig:burstgpt_readiness_sensitivity}
\end{figure}

\subsection{Timing Diagnostics for Two-Stage Recourse}
\label{app:two_stage_timing}

The Two-stage benchmark permits one post-launch deployment. Earlier recourse leaves more time for capacity to serve demand but uses less information; later recourse uses more observations but delays response and requires readiness to be preserved for longer. The training costs are not monotone in the recourse day, reflecting this trade-off. Day 4 has the lowest training cost, 12,575.05, and is fixed before the out-of-sample evaluation in Table~\ref{tab:burstgpt_main_results}.

\begin{table}[htbp]
\centering
\small
\caption{Training-sample selection of Two-stage recourse timing in the daily-review application}
\label{tab:two_stage_training_selection}
\renewcommand{\arraystretch}{1.15}
\begin{tabular}{lrrrrrr}
\toprule
\textbf{Recourse day} & 1 & 2 & 3 & 4 & 5 & 6 \\
\midrule
\textbf{Training cost} & 13671.05 & 13549.58 & 13226.07 & \textbf{12575.05} & 13902.74 & 13858.16 \\
\bottomrule
\end{tabular}
\vspace{3pt}
\begin{minipage}{0.97\textwidth}
\footnotesize
\emph{Notes.} The bold entry identifies the ex ante selected recourse day. The selected timing is fixed before out-of-sample evaluation.
\end{minipage}
\end{table}

Figure~\ref{fig:two_stage_timing} reports the information--delay--readiness trade-off across candidate recourse days.

\begin{figure}[htbp]
\centering
\includegraphics[width=3in]{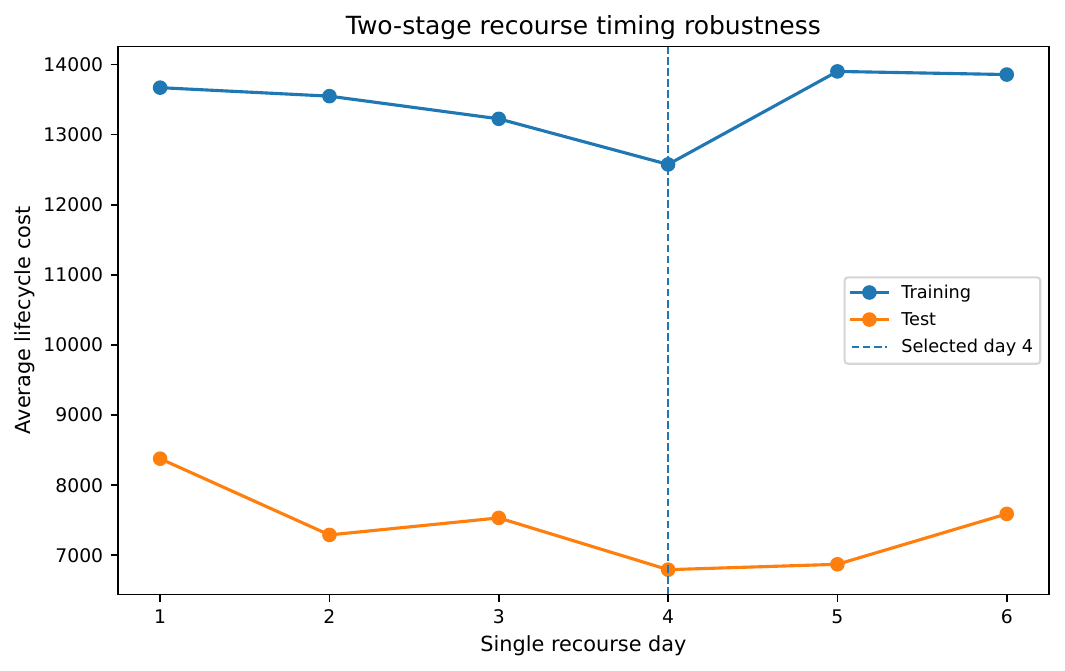}
\caption{Supplementary timing diagnostic for Two-stage recourse in the daily-review application. The diagnostic varies the fixed recourse day and reports the resulting lifecycle cost.}
\label{fig:two_stage_timing}
\end{figure}

Even after this ex ante timing choice, Two-stage recourse remains more costly than CBAP. The policy cannot adjust its response time or channel decision to the belief path observed in a particular lifecycle.

\end{document}